\documentclass[10pt]{article}

\usepackage[utf8]{inputenc}
\usepackage[T1]{fontenc}

\usepackage{fullpage}
\usepackage{changepage}

\usepackage{xcolor}
\usepackage{amsfonts,amsthm,amssymb,amsmath}
\usepackage{mathtools}
\usepackage{wasysym}
\usepackage{xspace}
\usepackage{graphicx}
\usepackage{breqn}
\usepackage{algorithm}
\usepackage{algorithmic}
\usepackage{tabularx}

\usepackage[english]{babel} 
\usepackage{caption}
\usepackage{subcaption}
\usepackage{paralist}
\usepackage{multirow}

\usepackage{hyperref}

\usepackage[noabbrev,capitalise]{cleveref}
\usepackage{autonum}

\newtheorem{thm}{Theorem}[section]
\newtheorem{prop}[thm]{Proposition}
\newtheorem{lem}[thm]{Lemma}
\newtheorem{claim}[thm]{Claim}
\newtheorem{cor}[thm]{Corollary}
\newtheorem{defn}[thm]{Definition}

\newtheorem*{ack*}{Acknowledgment}
\theoremstyle{remark}
\newtheorem{rem}{Remark}[section]

\newtheorem*{ex*}{Example}

\crefname{thm}{Theorem}{Theorems}
\crefname{lem}{Lemma}{Lemmas}
\newcommand{\ie}{\emph{i.e.}\xspace}
\newcommand{\eg}{\emph{e.g.}\xspace}

\newcommand*{\dual}[1]{{#1^*}}
\newcommand*{\switch}[1]{\widetilde{#1}}
\newcommand*{\unbud}[1]{#1^\circ}
\newcommand*{\transpose}[1]{\overline{#1}}
\newcommand*{\undir}[1]{\overline{#1}}
\newcommand*{\complementMap}[1]{#1^\complement}
\newcommand*{\cut}[1]{{\text{Open}\left(#1\right)}}

\DeclareMathOperator{\face}{face}
\DeclareMathOperator{\vertex}{vertex}
\DeclareMathOperator{\nextF}{NaF}
\DeclareMathOperator{\nextV}{NaV}
\DeclareMathOperator{\prevV}{PaV}
\DeclareMathOperator{\prevF}{PaF}
\DeclareMathOperator{\nextE}{NE}
\DeclareMathOperator{\prevE}{PE}
\newcommand*{\refine}[2]{{#1\preceq #2}}
\newcommand*{\coarsen}[2]{{#1\succeq #2}}

\floatname{algorithm}{Algorithm}

\title{
Blossoming bijection for higher-genus maps.}

\author{Mathias Lepoutre \thanks{\href{mailto:mathias.lepoutre@polytechnique.edu}{mathias.lepoutre@polytechnique.edu}
. Partially supported by the project ANR-16-CE40-0009.}}

\begin{document}

\maketitle
\begin{abstract}
In 1997, Schaeffer described a bijection between Eulerian planar maps and some trees. In this work we generalize his work to a bijection between bicolorable maps on a surface of any fixed genus and some unicellular maps with the same genus.
An important step of this construction is to exhibit a canonical orientation for maps, that allows to apply the same local opening algorithm as Schaeffer.

As an important byproduct, we obtain the first bijective proof of a result of Bender and Canfield from 1991, when they proved that the generating series of maps in higher genus is a rational function of the generating series of planar maps.

\textbf{keywords:} combinatorics, bijection, maps, higher genus, blossoming tree, rationality.
\end{abstract}

\tableofcontents

\section{Introduction} 
A \emph{map} of genus $g$ is a proper embedding of a graph in $\mathbb{S}_g$, the torus with $g$ holes. 
In addition to be rich combinatorial objects by themselves, maps have many links with various fields of algebra and mathematical physics (\eg \cite{LanZvo04, Eyna16}). The probabilistic approach of maps, leading to the definition of continuous surfaces such as the Brownian map, is also a very active domain. The structural study of maps is a deep subject, and it seems that it is always interesting to have a better understanding of maps, given the very diverse related topic.

Planar maps (or maps of genus 0) have been studied extensively since the pioneering work of Tutte in the sixties \cite{Tutt63}. 
In a series of work, Tutte obtained remarkable formulas for many families of maps. 
His techniques relies on some recurrence relations for maps, obtained through combinatorial decomposition, and some clever manipulations of generating series. 
They were extended in the late eighties to the case of maps with higher genus by Bender and Canfield, who first obtained the asymptotic number of maps on any orientable surface of genus $g$ \cite{BenCan88} and then obtained in~\cite{BenCan91} in 1991 the following stronger result:

\begin{thm}[Bender and Canfield \cite{BenCan91}]\label{thm:BenCan91}
For any $g\geq 0$, the generating series $M_g(z)$ of maps of  genus $g$ enumerated by edges is a rational function of $z$ and $\sqrt{1-12z}$. 
\end{thm}

The enumerative results obtained using Tutte's techniques show some underlying very strong structural properties of maps, and call for bijective explanations. 
The first such explanation was the bijection of Cori and Vauquelin \cite{CorVau81}. 
Indeed, the enumerative formula of planar maps obtained by Tutte has a very simple closed form, that Cori and Vauquelin were the first to explain bijectively in 1981. 
This work was soon followed by many others, starting with the pioneering work of Schaeffer, in the late 90's, and was the beginning of the bijective combinatorics of maps.

In this vein, the purpose of this paper is to give a bijective explanation of enumerative results in higher genus. 
In particular, our main result is the first bijective proof of \cref{thm:BenCan91}, for $g\geq 2$. 

In the planar case, Schaeffer exhibits in~\cite{Scha97} a constructive bijection between Eulerian planar maps and some so-called \emph{blossoming trees}. 
The blossoming tree associated to a map is one of its spanning trees, decorated by some \emph{stems}, that enable to reconstruct the ``missing edges''. 
Our work is a generalization of \cite{Scha97} to maps of any genus.

In genus $g>0$, the natural counterpart of trees are unicellular maps (\ie maps with only one face) and we obtain in this work the following result (the terminology is introduced in \cref{subsec:closure}):

\begin{thm}\label{thm:mainGen}
There exists a constructive weight-preserving and genus-preserving bijection between rooted bicolorable maps and well-rooted well-labeled well-oriented unicellular blossoming maps.
\end{thm}

Thanks to this theorem, the enumeration of maps boils down to the much easier enumeration of this specific family of unicellular blossoming maps.
Using techniques used in particular by Chapuy, Marcus, and Schaeffer in \cite{ChMaSc09}, we are able to decompose these unicellular maps into a \emph{scheme} with \emph{branches}. Similarly to \cite{ChMaSc09}, the proof of \cref{thm:BenCan91} then amounts to showing a certain symmetry, that we are able to prove. 

Let us now put our work in context of the existing literature. In the planar case, there are numerous bijections between maps and some families of decorated trees. Two main trends emerge in these bijections. Either the decorated trees are some blossoming trees as already described (\eg \cite{CorVau81,Scha97,BoDiGu02,PouSch06}) or the trees are decorated by some integers that capture some metric properties of the maps (\eg \cite{Scha98,BoDiGu04}).
Bijections of the latter type have been successfully extended to higher genus~\cite{ChMaSc09,Chap10,Mier09}, and to non-orientable surface (\cite{ChaDol17,Bett16}). With these techniques (in particular, see \cite{ChMaSc09}), it is possible to show that the generating series of maps can be expressed as a rational function of some auxiliary functions, whose degree of algebraicity, unfortunately, is higher than the known enumerative results. 
The situation is much different in the case of bijections with blossoming trees, and apart from the recent work~\cite{DeGoLe17} which presents a bijection between simple triangulations of genus 1 (with some additional constraints) and a family of blossoming unicellular maps, there was, previously to our work, which is a generalization of \cite{Scha97}, no other extension of the existing planar bijections. 

Let us continue with an important connection to our work. 
As emphasized by Bernardi~\cite{Bern06} in the planar case and generalized by Bernardi and Chapuy~\cite{BerCha11}, a map endowed with a spanning unicellular embedded graph (whose genus can be smaller than the genus of the initial surface) can also be viewed as a map endowed with an orientation of its edges with specific properties. 
The general theory of $\alpha$-orientations developed by Felsner in the planar case~\cite{Fels04} has been successfully combined with the result of~\cite{Bern06} to give general bijective schemes in the planar case~\cite{BerFus12a,BerFus12b,AlbPou15}, which enables to recover the previously known bijections. 
It would be highly desirable to obtain systematic bijective schemes in higher genus by combining Bernardi and Chapuy's result together with the theory of $c$-orientations introduced by Propp~\cite{Prop93} or its extension by Felsner and Knauer~\cite{FelKna09}.
The main difficulty to tackle is to characterize the orientations that produce spanning unicellular embedded graph whose genus matches the genus of the original surface. The orientation we choose in our work does produce such embedded graphs, and our work can hence be seen as an important step in that direction.

The bijection of Schaeffer \cite{Scha97} was extended by Bouttier, di Francesco and Guitter \cite{BoDiGu02}, so as to work on any map, instead of only Eulerian ones. This work was then revisited by Albenque and Poulhalon \cite{AlbPou15}, whose general framework allows to see the bijection of \cite{BoDiGu02} as the opening of a map, endowed with a well-chosen fractional orientation. 
In \cref{sec:nonBic}, we generalize these extensions to maps on surfaces of any genus, so as to get a direct bijection between general maps and some unicellular blossoming maps.

Our main result deals with the rationality of bicolorable $4$-valent maps. 
It would be very interesting to have similar rationality results for general bicolorable maps, with a control on the degrees of the vertices, for instance by giving a rational parametrization of bicolorable maps of degree at most $6$. 
This direction will be explored in a future work.

In \cite{BeCaRi93}, Bender, Canfield and Richmont generalized the work of \cite{BenCan91}, by enumerating maps by vertices and faces instead of edges only. 
Similarly, they prove the existence of a rational parametrization of the generating series in higher genus.
It would be really interesting to generalize our work so as to obtain a combinatorial explanation of this more general rational parametrization. 
This is another direction that we wish to explore in future work.

To conclude, note that there are a lot of other more precise structural or enumerative properties of maps and related objects that can be proved using involved mathematical studies and calculation \cite{ChaDol17,Eyna16}, but still call for bijective or combinatorial explanation. From this perspective, the present work is still at the very beginning of the understanding of maps we hope to reach through combinatorics.

{\bf Organization of the paper: }In \cref{sec:or}, we recall definitions about maps and orientations and state Propp's theorem adapted to our setting. In \cref{sec:bijection}, we define an explicit and constructive bijection between ($4$-valent or not) bicolorable maps and a family of unicellular maps. In \cref{sec:analysis}, we analyse this family by reducing these maps to schemes, reducing the proof of \cref{thm:BenCan91} to the rationality of a restricted family of maps, that all have the same \emph{scheme}, or alternatively, the symmetry of this series in terms of an intermediate series. In \cref{sec:symmetry} we prove this symmetry by doing some additional work based on surjections surjections. Finally in \cref{sec:nonBic} we present an extension of our main bijection to general maps (not necessarily bicolorable), using fractional orientations.

{\bf Notation:} In this article, combinatorial families are named with calligraphic letters, their generating series is the corresponding capital letter, and an object of the family, is usually denoted by the corresponding lower case letter. The size being denoted by $|\centerdot|$, we therefore have for a combinatorial family $\mathcal{S}$: $S(z)=\sum_{s\in\mathcal{S}}z^{|s|}$. \\[-.6cm]

\section{Orientations in higher genus}
\label{sec:or}
\subsection{General}
\label{subs:def}

We begin with some definitions about maps. 
\begin{defn}[embedded graph, map]
An \emph{embedded graph} is an embedding of a connected graph into a given compact surface, taken up to orientation-preserving homeomorphisms	 of the surface. 
An embedded graph is \emph{cellularly embedded} if all its \emph{faces} (connected component of the complement) are homeomorphic to discs. 
A \emph{map} is a cellularly embedded graph. 
\end{defn}
The set of maps, counted by number of edges, is denoted~$\mathcal{M}$.
In this paper we only consider maps embedded on orientable surfaces.
\emph{General maps} have no other restriction, and in particular, can have loops or multiple edges.

A map on an orientable surface can alternatively be defined as a graph equipped with a cyclic order on edges around each vertex.

\begin{defn}[genus]
The \emph{genus} of a map is the genus of its underlying surface.
The genus of an embedded graph is the genus of the map obtained by replacing each face of the embedded graph by a disc. 
\end{defn}
The genus of an embedded graph is lower or equal to the genus of its underlying surface. The two are equal if and only if the embedded graph is a map.
All families of maps can be refined by their genus; we denote this refinement by an index indicating the topological genus, so that for instance $\mathcal{M}_0$ is the set of maps in the sphere. See \cref{fig:mapGenus1} for an example of map of genus $1$.

\begin{defn}[corner, degree]
An adjacency between a face and a vertex is called a \emph{corner}. 
Note that a single pair vertex-face can give rise to several distinct corners.
The \emph{degree} of a face (resp.~vertex) is the number of adjacent corners. 
\end{defn}
\begin{defn}[rooting]
In what follows, to get rid of automorphisms, we always assume that the maps we consider are \emph{rooted}, meaning that a corner, called the\emph{root corner}, that we indicate by a brown arrow, is distinguished.
The vertex and face adjacent to the root corner are called \emph{root vertex} and \emph{root face}.
\end{defn}

\begin{figure}
  \centering
  \begin{subfigure}[b]{.45\textwidth}
   \centering\includegraphics[width=\textwidth,page=8]{exampleMap.pdf}
   \vspace*{-42pt}
	  \caption{\label{fig:mapGenus1}A rooted map of genus~1.}
  \end{subfigure}
  \quad
  \begin{subfigure}[b]{.45\textwidth}
   \centering\includegraphics[width=\textwidth]{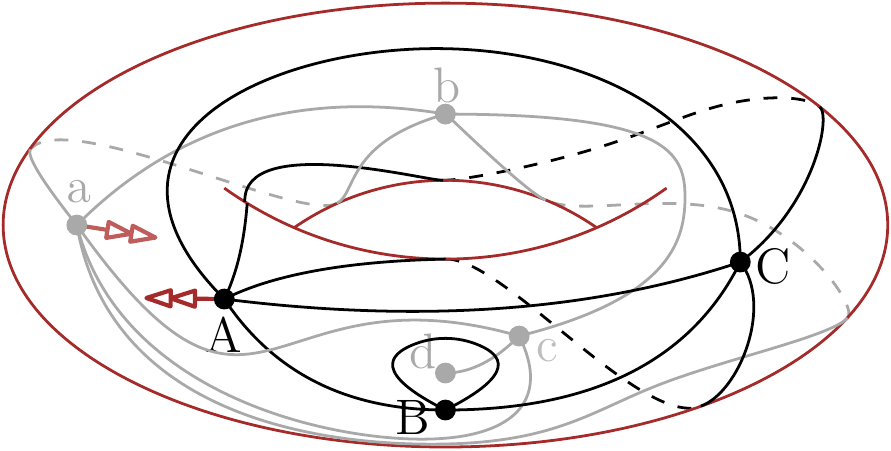}
   \caption{\label{fig:dualMap}The same map (grey) and its dual (black).}
  \end{subfigure}
  \caption{Examples of maps. The root corner is indicated by a double-arrow.}
  \label{fig:mapEx}
\end{figure}

The set of vertices (resp.~edges, faces) of $m$ is denoted $\mathcal{V}(m)$ (resp.~$\mathcal{E}(m)$, $\mathcal{F}(m)$). 
The number of vertices (resp.~edges, faces) of $m$ is denoted $v(m)$ (resp.~$e(m)$, $f(m)$).
These notation can also be specified by degree, so that for instance $f_k(m)$ is the number of degree-$k$-faces of $m$.
The genus of $m$ is denoted $g(m)$. 
We recall Euler's formula: 

\begin{prop}[Euler's formula]
For any map $m\in{M}$, $v(m)-e(m)+f(m)=2-2g(m)$.
\end{prop}

\begin{defn}[dual map]
Since an edge connects two vertices and separates two faces, we can define the \emph{dual map} $\dual{m}$ of $m$ by exchanging the role of vertices and faces, and swapping the connection and separation induced by each edge (see \cref{fig:dualMap}). 
The root corner remains the same (but its vertex and its face are exchanged). 
\end{defn}
Note that duality is involutive:~$\dual{(\dual{m})}=m$.

\begin{defn}[unicellular, tree]
An embedded graph is called \emph{unicellular} if it has only one face. 
A \emph{tree} is an embedded graph with no cycle. 
\end{defn}
Trees are unicellular and have genus $0$.
Note that in genus $0$, any unicellular embedded graphs is a tree, whereas in a positive-genus surface, a unicellular embedded graph may have any genus lower than the genus of the surface.

\begin{defn}[bipartite map, bicolorable map]
A map is \emph{bipartite} if its underlying graph is bipartite, which means that its vertices can be \emph{properly} (so that no $2$ adjacent vertices have same color) colored black and white.
Dually, a map is \emph{bicolorable} if its faces can be \emph{properly} (so that no $2$ adjacent faces have same color) colored black and white.
\end{defn} 
Note that in particular, a bipartite map has no loop.
Note also that the faces of a bipartite map and the vertices of a bicolorable map necessarily have even degree. 
The set of bipartite (resp. bicolorable) maps counted by number of faces (resp. vertices) is denoted $\mathcal{BP}$ (resp. $\mathcal{BC}$). 
The generating series of bipartite and bicolorable maps, $BP(z)$ and $BC(z)$, can be refined in the following way: \[BP(\textbf{z})=BP(z_1,z_2,\cdots)=\sum\limits_{m\in\mathcal{BP}}\prod\limits_{k=1}^{\infty}z_k^{f_{2k}(m)}
\stackrel{\text{\tiny duality}}{=}
\sum\limits_{m\in\mathcal{BC}}\prod\limits_{k=1}^{\infty}z_k^{v_{2k}(m)}=BC(z_1,z_2,\cdots)=BC(\textbf{z}).\]

\begin{rem}\label{rem:eulMap}
A map is called \emph{Eulerian} if all its vertices have even degree. 
Note that bicolorable maps are Eulerian, and in fact the notions are equivalent in genus $0$.
However this is not the case in higher genus, where some additional non-local constraints appear along non-contractible cycles.
Though Eulerian is a more common property for graph, it seems that, in our setup, bicolorability is a more relevant map property, in particular in view of \cref{prop:radMap}.
\end{rem}

\begin{defn}[quadrangulation, $4$-valent map]
A map is called a \emph{quadrangulation} if all its faces have degree $4$. 
Dually, a map is \emph{$4$-valent} if all its vertices are of degree $4$. 
\end{defn}
The set of bipartite quadrangulations (resp.~bicolorable $4$-valent maps), counted by number of faces (resp. number of vertices), is denoted~$\mathcal{BP}^\square$~(resp.~$\mathcal{BC}^\times$).
Their generating series therefore satisfy: $BP^\square(z)=BP(0,z,0,0,\cdots)\stackrel{duality}{=}BC(0,z,0,0,\cdots)=BC^\times(z)$.

\begin{prop}
\label{prop:radMap}
 General maps of genus $g$ with $n$ edges are in bijection with $4$-valent bicolorable maps of genus $g$ with $n$ vertices, or dually, with bipartite quadrangulations of genus $g$ with $n$ faces. Therefore, $M_g(z)=BC_g^\times(z)=BP_g^\square(z)$.
\end{prop}

\begin{proof}
Starting from a map $m$, we construct bijectively as follows a $4$-valent bicolorable map called the \emph{radial map} and denoted $r$. 
We create a vertex in $r$ for each edge of $m$. 
For each corner of $m$, we then add an edge in $r$ between the two vertices corresponding to the edges of $m$ adjacent to the chosen corner. 
Out of the two corners of $r$ corresponding to the root corner of $m$, we choose the leftmost one as the root corner of $r$. 
See \cref{fig:exRad} for an example (the orientation of the edges will be explained in \cref{subs:orientations}).
\end{proof}

Note that the radial map is the dual of the so-called \emph{quadrangulated map} (see \cref{fig:exQuad}).

\begin{figure}
  \centering
  \begin{subfigure}[c]{.46\linewidth}
   \centering
   \includegraphics[scale=1,page=7]{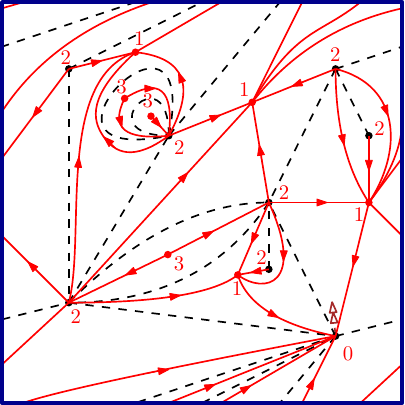}
   \caption{A map (dashed black) with its (bipartite quadrangulation) quadrangulated map (full red), with geodesic orientation.}\label{fig:exQuad}
  \end{subfigure} \hfill
  \begin{subfigure}[c]{.46\linewidth}
   \centering
   \includegraphics[scale=1,page=6]{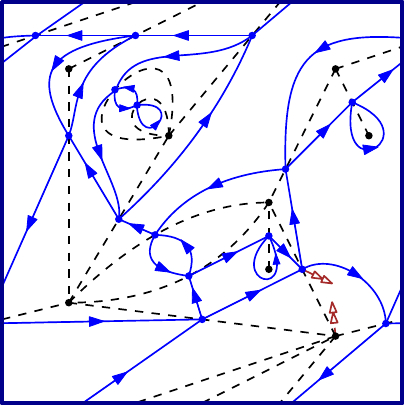}
   \caption{A map (dashed black) with its ($4$-valent bicolorable) radial map (full blue), with dual-geodesic orientation.}\label{fig:exRad}
  \end{subfigure}
  \caption{Classical constructions on a toroidal map.}
\end{figure}

\subsection{Structure of orientations of a graph}
\label{subs:orientations}

\begin{defn}[orientation, dual orientation]
An \emph{orientation} of a map is an orientation of each of its edges. 
The \emph{dual orientation} $\dual{o}$ of an orientation $o$ of a map $m$ is the orientation of $\dual{m}$ where all dual edges are oriented from the face to the right of the primal edge toward the face to its left.
\end{defn}
Note that applying duality twice reverses the orientation (duality on oriented maps is not involutive).

Orientations provide additional structural properties to maps, useful for algorithmic purposes. 
However, since our final purpose is to study maps without an orientation, it is convenient to assign a canonical orientation to maps. 
Such an orientation will be provided in \cref{cor:lat}, and will be obtained as the minimum of a lattice of orientations, as described below.

\begin{defn}[bipartite orientation, vertex-push]
An orientation of a map is called \emph{bipartite}, if it has as many forward edges as backward edges along any cycle (note that any cycle of a bipartite map is made of an even number of edges).
The set of bipartite orientations is endowed with the \emph{vertex-push} operation (see \cref{fig:exVertPush}), that changes a sink distinct from the root into a source, by reversing all adjacent edges. 
\end{defn}

A map that has a bipartite orientation is necessarily bipartite.

\begin{defn}[bicolorable orientation]
An orientation of a map is called \emph{bicolorable} if its dual orientation is bipartite.
\end{defn}
In other words, each dual cycle has as many edges crossing to the right as edges crossing to the left.
A map that has a bicolorable orientation is necessarily bicolorable.
\begin{rem}\label{rem:eulOr}
A bicolorable orientation is Eulerian, meaning that all vertices have equal indegree and outdegree. 
However, again, the $2$ notions are not equivalent on a surface of positive genus, but this is not only due to \cref{rem:eulMap}
Even if the map is bicolorable, an Eulerian orientation is not necessarily bicolorable, because of the existence of some non-contractible dual cycles inducing additional non-local constraints for bicolorability.
\end{rem}

\begin{defn}[face-flip]
The set of bicolorable orientations is endowed with the operation dual to vertex push, called the \emph{face-flip}.
\end{defn}
\begin{rem}
Face-flips can alternatively be defined in the following way (see \cref{fig:exFaceFlip}): take a clockwise face distinct from the root, and change the orientation of all edges adjacent to that face.
\end{rem}

\begin{figure}
  \centering
  \begin{subfigure}[c]{.46\linewidth}
   \centering
   \includegraphics[width=\linewidth,page=2]{exFaceFlip}
  \label{fig:exVertPush}
   \caption{A vertex-push.}
  \end{subfigure} \hfill
  \begin{subfigure}[c]{.46\linewidth}
   \centering
   \includegraphics[width=\linewidth,page=1]{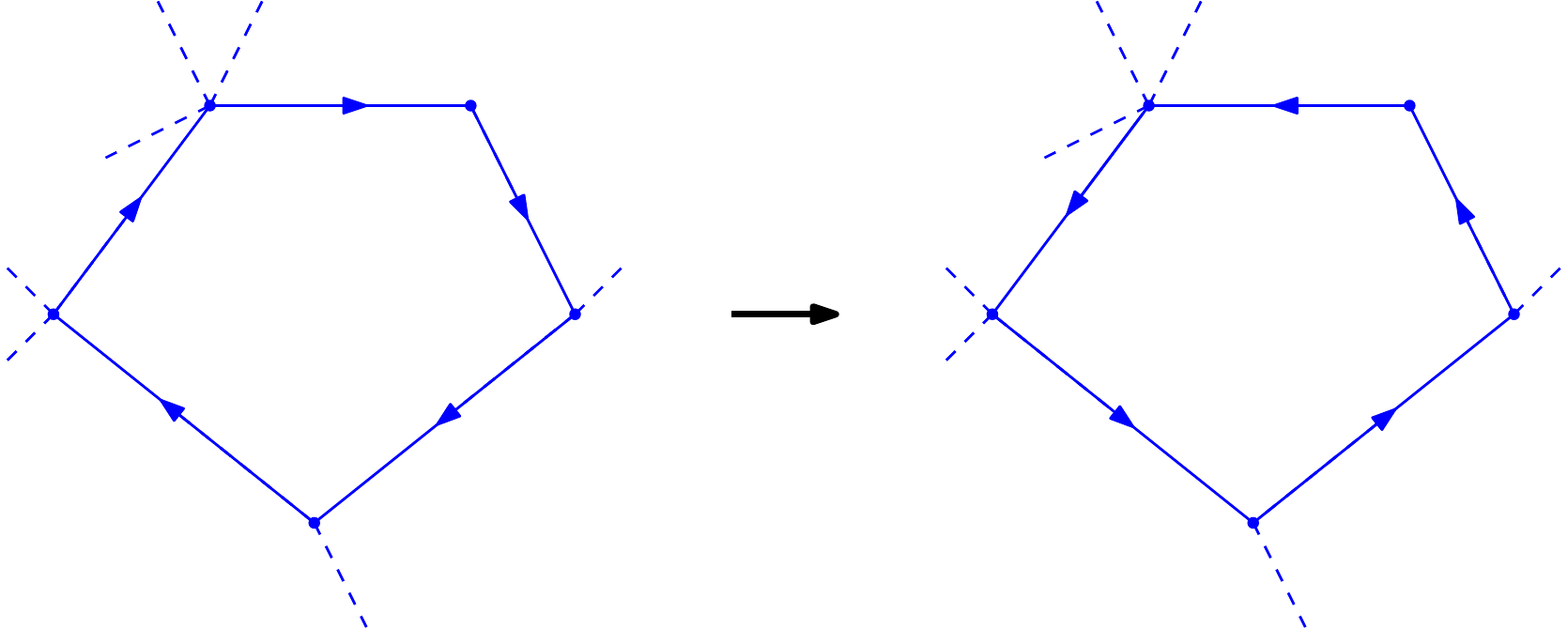}
   \caption{A face-flip}
  \label{fig:exFaceFlip}
  \end{subfigure}
  \caption{Some operations on orientations of maps.}
\end{figure}

Vertices of a bipartite map can be labeled by their distance to the root. Since the map is bipartite, two adjacent vertices cannot have the same label. 
\begin{defn}[geodesic orientation, dual-geodesic orientation]
The \emph{geodesic orientation} of a bipartite rooted map is the orientation whose edges are all oriented towards their extremity with smaller label.

The dual of the geodesic orientation is called the \emph{dual-geodesic orientation} (see \cref{fig:exRad} for an example).
\end{defn}

Along any cycle, forward (resp.~backward) edges in the geodesic orientation correspond to a label increasing (resp.~decreasing) by exactly $1$. Therefore, the geodesic orientation is bipartite, and the dual-geodesic orientation is bicolorable. 

The next result directly follows from \mbox{\cite[Theorem~1]{Prop93}}.

\begin{thm}[Propp]
The transitive closure of the vertex-push operation endows the set of  bipartite orientations of a fixed bipartite map with a structure of distributive lattice.
\end{thm}

In particular, this means that this set has a unique minimum for vertex-push. By definition, the only sink of the geodesic orientation is the root vertex, which means that the geodesic orientation is minimal for the vertex-push operation, and by consequence:

\begin{cor}
The minimum of the above-mentioned lattice of bipartite orientation of a map is the geodesic orientation of this map.
\end{cor}

Furthermore, we obtain by duality:

\begin{cor}
The transitive closure of the face-flip operation endows the set of bicolorable orientations of a fixed bicolorable map with a structure of distributive lattice, and its minimum is the dual-geodesic orientation.
\end{cor}

We can therefore characterize uniquely the dual-geodesic orientation of a given bicolorable map:

\begin{cor}
\label{cor:lat}
The dual-geodesic orientation of a bicolorable map is the unique bicolorable orientation of this map with no clockwise face.
\end{cor}

\begin{rem}
In recent attempts to extend orientations on maps of higher genus, the notion of $\alpha$-orientations, due to Felsner \cite{Fels04}, has been used (\eg in \cite{GoKnLe16}).
This leads to study Eulerian orientations (see \cref{rem:eulOr}) instead of bicolorable orientations. 

Unfortunately, Eulerian orientation with the face-flip operation give rise to several unconnected lattices. A classical approach would be to canonically select one of these connected component, and only work on this one. A natural choice for such a component is the set of bicolorable orientations, which are indeed a strict subset of Eulerian orientation. Therefore, defining bicolorable orientations in the first place seems more convenient for our purposes. 
\end{rem}

In the rest of this paper, we will use orientation of maps as an additional layer of information, useful for algorithmic purposes, but determined in a canonical way using \cref{cor:lat}.

\section{Closing and opening maps}
\label{sec:bijection}

In this section, we describe an algorithm called the \emph{opening algorithm}, that starts from an oriented map and creates a unicellular blossoming oriented map whose \emph{closure} is the original map. The planar version of this algorithm was first described by Schaeffer in \cite{Scha97}. It was then generalized by Bernardi in \cite{Bern06}, and Bernardi and Chapuy in \cite{BerCha11} for the higher-genus case, to a different setup where maps come with an orientation. This is slightly different from this work since we give a canonical orientation to maps. See \cref{rem:BerCha} for a more detailed discussion of the differences with this work.

\subsection{Blossoming maps and their closure}

\begin{defn}[blossoming map, bud, leaf]
A \emph{blossoming} map $b$ is a map with additional stems attached to its corners. These stems are oriented and hence can be of two types; an outgoing stem is called a \emph{bud}, while an ingoing stem is called a \emph{leaf}. We require that a blossoming map has as many buds as leaves. 
Blossoming maps are always assumed to be rooted on a bud. 
\end{defn}

\begin{defn}[interior map]\label{def:interiorMap}
The \emph{interior map} of a blossoming map $b$, denoted $\unbud{b}$, is the map obtained from $b$ by removing all its stems. 
\end{defn}

Most blossoming maps we usually consider are oriented, which leads to these additional definitions:

\begin{defn}[blossoming degrees]
In a blossoming oriented map, the \emph{interior degree} (resp.~\emph{blossoming degree}, resp.~\emph{degree}) of a vertex is the degree of this vertex in the interior map (resp.~the number of stems attached to it, resp.~the sum of the interior and blossoming degrees). These can all be refined into \emph{ingoing} and \emph{outgoing} degrees. 
\end{defn}

As stated in \cref{thm:mainGen}, unicellular blossoming maps are instrumental to our approach, because they can encode maps, while being easier to analyse. To describe the bijection mentioned in \cref{thm:mainGen}, called the \emph{closing algorithm}, we first introduce the \emph{contour word} of a blossoming unicellular map.

\begin{defn}[contour word]
Let $b$ be a blossoming unicellular map. The \emph{contour word} of $b$ is the word on $2$ letters $U$ and $D$ defined as follows: when doing a clockwise tour of the unique face (which means that the face is on the right), starting from the root bud, write $U$ (for up-step) for each bud and $D$ (for down-step) for each leaf. The contour word can naturally be seen as a $1$-dimensional walk with up- and down-steps, starting and ending at height $0$.
\end{defn}

We describe in \cref{alg:closure} how a unicellular blossoming map can be \emph{closed} into a general map (see \cref{fig:closure} for a planar example, and \cref{fig:algo} from right to left for a genus-$1$ example). The result of the closing algorithm applied to a map $b$ is called the \emph{closure} of $b$ and denoted $Close(b)$.

\begin{algorithm}[h]
\caption{the closing algorithm}
\label{alg:closure}
Let $b$ be a unicellular blossoming map. 

We write the contour word of $b$ and match its steps by pairs upstep/downstep: each up-step $U$ going from height $i$ to $i+1$ is matched to the first down-step $D$ after $U$ going from height $i+1$ to $i$. 

This is done in a cyclic manner, meaning that the last upstep of the contour word going from height $i$ to height $i+1$ is matched with the first down-step $D$ going from height $i+1$ to height $i$. 

The stems corresponding to matched steps are then merged into a single oriented edge. 

The new map is rooted on the corner just on the right of the edge formed by the former root bud, around the root vertex.
\end{algorithm}

\begin{figure}[h]
\vspace{-10pt}
\centering
\begin{subfigure}[c]{.48\linewidth}
\includegraphics[width=\textwidth,page=1]{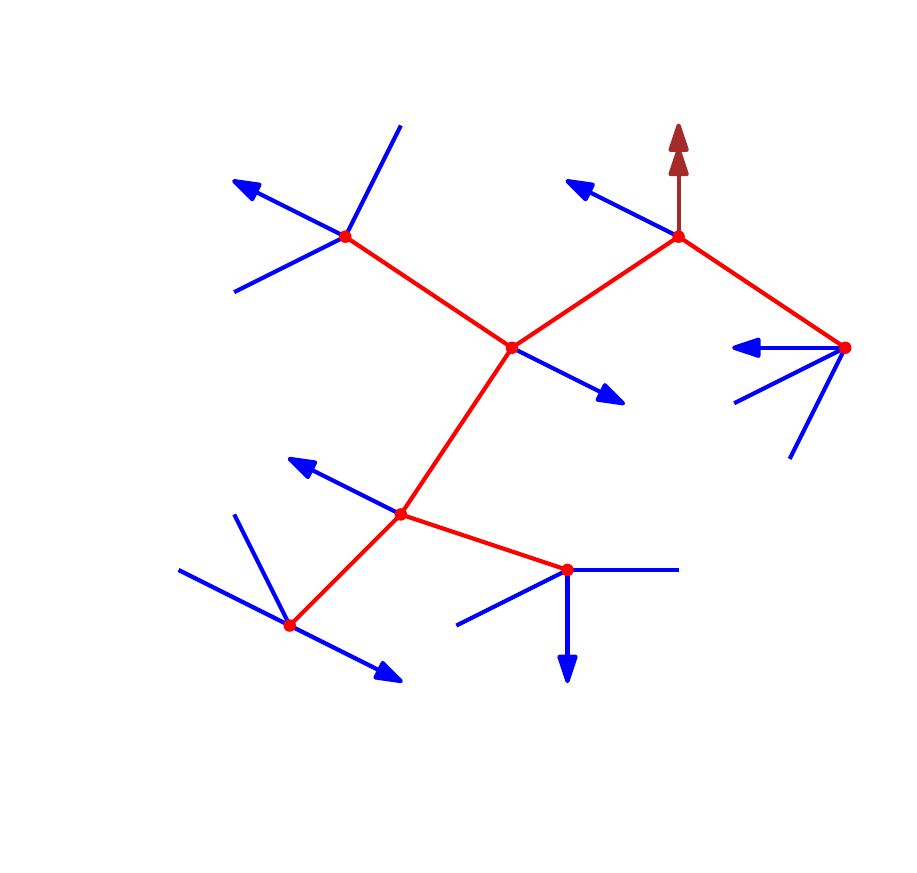}
\end{subfigure}
\begin{subfigure}[c]{.48\linewidth}
\flushright
\includegraphics[width=\textwidth,page=19]{exClosureTree}
\end{subfigure}
\vspace{-40pt}
\caption{The closure of a (well-rooted) blossoming tree.}
\label{fig:closure}
\end{figure}

Note that the way stems are matched, which is similar to a well-parenthesizing matching, implies that the created edges are non-crossing. 
Note also that the cyclic definition of the closing algorithm means that the matching does not depend on the root.
This implies that several blossoming unicellular maps can lead to the same map, up to the position of the root. 
However this is not the case anymore if we restrict the way a blossoming map can be rooted.

\begin{defn}[well-rooted]\label{def:wroot}
A map $b$ is called \emph{well-rooted} if its contour word is a Dyck word.
\end{defn}

\begin{rem}\label{rem:wr}
When applying the closing algorithm to a well-rooted map, since the contour word is a Dyck path, the cyclic definition of the algorithm is not needed, which implies that all closing edges have the root on their right.
\end{rem}

\begin{defn}[rootable stem, well-rootable stem]
A stem is called \emph{rootable} if it is either a leaf or the root bud.

Suppose we change the root up-step of the contour word of a map $b$ into a down-step. The contour word now goes from height $0$ to height $-2$, and its minimum height is denoted $-k$. 
The first step going from height $-k+2$ to height $-k+1$ and the first step going from height $-k+1$ to height $-k$ are called \emph{well-rootable steps/stems}. 
\end{defn}

Note that well-rootable stems are rootable, and that if we apply a cyclic permutation to a contour word, the set of its well-rootable steps remains the same.

\begin{rem}
A map $b$ is \emph{well-rooted} if and only if its root bud is well-rootable.
\end{rem}

\begin{defn}[undirected map, root-equivalence, unrooted map]\label{def:unroot}
The \emph{undirected map} of a blossoming map is the map obtained by forgetting the orientation of both the edges (if relevant) and the stems.
Two rooted blossoming unicellular maps are called \emph{root-equivalent} if they have the same undirected map and the same set of rootable stems (in particular they do not necessarily have the same root).
The \emph{unrooted map} $\undir{b}$ of $b$ is the equivalence class of $b$ for root-equivalence.
\end{defn}

\begin{rem}\label{rem:unroot}
The blossoming orientation of $b$ can be easily recovered from $\undir{b}$ if we know which rootable stem of $\undir{b}$ is the root of $b$.
\end{rem}
\begin{rem}\label{rem:wroot}
Note that it is possible to know which rootable stems of an unrooted map are well-rootable. As a consequence, a well-rooted map can alternatively be seen as unrooted map with a distinguished well-rootable stem. This point of view will be useful in \cref{subs:wroot}.
\end{rem}

\subsection{The opening algorithm}

Given a rooted oriented map $m$, we describe the \emph{opening algorithm} as follows (a more rigorous description will be given later in \cref{alg:opening}). We explore the map starting from the root. When we meet an unexplored edge, if it is ingoing, we follow it, if it is outgoing, we cut it and replace it by a bud. When we meet an already explored edge, it was either followed, in which case we follow it back, or cut, in which case we just add a leaf. We stop when we get back to the root. The resulting blossoming map is called the \emph{opening} of $m$ and denoted $Open(m)$.
A planar example of an execution of the opening algorithm is given in \cref{fig:exCornerOpening}, and a genus-$1$ example is given in \cref{fig:algo}, from left to right. 

\begin{figure}
\vspace{-20pt}
\centering
\begin{subfigure}[c]{.48\linewidth}
\centering
\includegraphics[width=\textwidth,page=6]{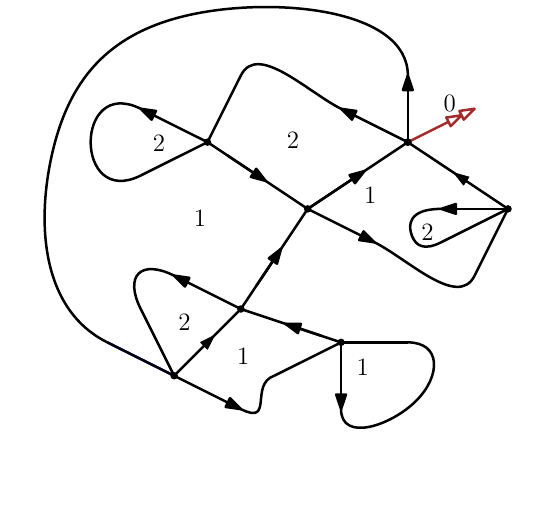}
\end{subfigure}
\begin{subfigure}[c]{.48\linewidth}
\centering
\includegraphics[width=\textwidth,page=14]{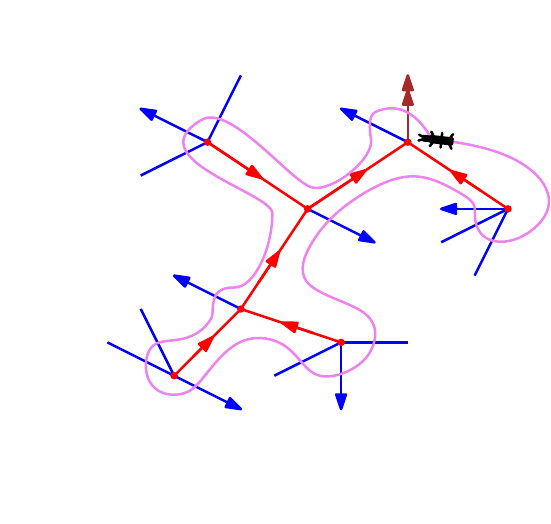}
\end{subfigure}
\vspace{-40pt}
\caption{The opening of a planar map with dual-geodesic orientation.}
\label{fig:exCornerOpening}
\end{figure}

More formally, we define the algorithm as a walk in the \emph{corner map}.
\begin{defn}[corner map]
Recall that a \emph{corner} $c$ is an adjacency between a face and a vertex, that we respectively denote $\face(c)$ and $\vertex(c)$. 
We define two permutations on the set of corners. If $c$ is a corner, $\nextF(c)$ is the next corner around $\face(c)$ in clockwise order, while $\nextV(c)$ is the next corner around $\vertex(c)$ in counter-clockwise order. 
The inverse permutations are naturally called $\prevF(c)$ and $\prevV(c)$. 
A corner is delimited by two edges: $\nextE(c)$ joins $\vertex(c)$ and $\vertex(\nextF(c))$ and separates $\face(c)$ and $\face (\nextV(c))$, while $\prevE(c)$ joins $\vertex(c)$ and $\vertex(\prevF(c))$ and separates $\face(c)$ and $\face(\prevV(c))$.

The \emph{corner map} of a map $m$ is the oriented map whose vertices are the corners of $m$ and which has for any corner $c$ an edge from $c$ to $\nextV(c)$ and an edge from $c$ to $\nextF(c)$.
\end{defn}
The corner map is an Eulerian $4$-valent graph, endowed with an Eulerian orientation.
These definitions can be visualized in \cref{fig:defCorner}.

A formal definition of the opening algorithm, seen as an oriented walk on the corner map, is given in \cref{alg:opening}, and illustrated in \cref{fig:exCornerOpening}.

\begin{algorithm}[h]
\caption{the opening algorithm}
\label{alg:opening}
\begin{algorithmic} 
\REQUIRE A map $m$ embedded on a surface $\mathcal{S}$, rooted at a corner $c_0$, along with its dual-geodesic orientation.
\ENSURE An oriented blossoming embedded graph $b=\text{open}(m)$, embedded on $\mathcal{S}$.
\STATE Set $c=c_0$, $b=\emptyset$, and $E_V=\emptyset$ ($E_V$ is the set of visited edges). 
\REPEAT
\STATE $e=\nextE(c)$.
\IF{$e\notin E_V$ and $e$ is oriented toward $\vertex(c)$}
\STATE add $e$ to $E_V$
\STATE add $e$ to $b$
\STATE $c\leftarrow\nextF(c)$
\ELSIF{$e\notin E_V$ and $e$ is outgoing from $\vertex(c)$}
\STATE add $e$ to $E_V$
\STATE Add a bud to $b$ in place of $e$.
\STATE $c\leftarrow\nextV(c)$
\ELSIF{$e\in E_V$ and $e$ is oriented toward $\vertex(c)$}
\STATE Add a leaf to $b$ in place of $e$.
\STATE $c\leftarrow\nextV(c)$
\ELSIF{$e\in E_V$ and $e$ is outgoing from $\vertex(c)$}
\STATE $c\leftarrow\nextF(c)$
\ENDIF
\UNTIL {$c=c_0$}
\RETURN $b$.
\end{algorithmic}
\end{algorithm}

This alternative definition highlights well the symmetry between the roles of faces and vertices, which we express in \cref{lem:dualCut}, and illustrate in \cref{fig:exDualOpening}. 
In addition to \cref{def:interiorMap}, the following two definitions are needed.
\begin{defn}[reflected map]
To a map $m$ we associate a \emph{reflected map} $\switch{m}$ which is the same as $m$ except that we switch the orientation of the underlying surface, which amounts to exchanging clockwise and counterclockwise, left and right.
\end{defn}
\begin{defn}[complement submap]
To a subgraph $s$ of a graph $g$ we associate the \emph{complement subgraph} $\complementMap{s}$ defined with the same set of vertices along with all edges in $g$ but not in $s$. This definition is naturally extended to the complement of a map by preserving the embedding. 
\end{defn}

\begin{lem}
\label{lem:dualCut}
Up to a change of orientation of the surface, a map and its dual yield complement submaps by the opening algorithm:
\[\unbud{\cut{m}}=\complementMap{\left(\unbud{\cut{\switch{\dual{m}}}}\right)}. \]
\end{lem}

\begin{figure}[h]
\centering
\begin{subfigure}[c]{.39\linewidth}
\centering
\includegraphics[width=\textwidth]{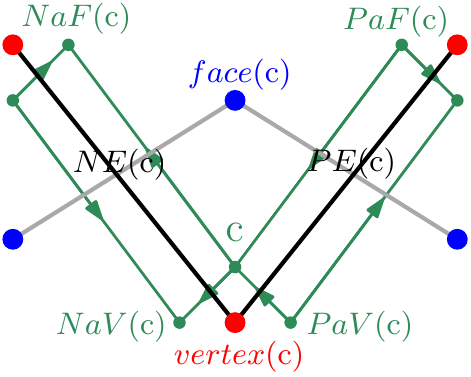}
\caption{A partial map, with red vertices, blue faces, black edges, and grey dual edges, along with its oriented corner map, in green.}
\label{fig:defCorner}
\end{subfigure}
\begin{subfigure}[c]{.60\linewidth}
\flushright
\includegraphics[width=0.9\textwidth,page=30]{exClosureTree}
\caption{The opening algorithm and its reflected dual yield complement maps.}
\label{fig:exDualOpening}
\end{subfigure}
\caption{The opening algorithm is a walk on the corner map.}
\end{figure}

\begin{proof}
Let $m$ be a map, $c$ a corner of $m$, and $c'$ the same corner, in $\switch{\dual{m}}$. It is easy to see that $\nextF(c)$ and $\nextV(c')$ (resp. $\nextV(c)$ and $\nextF(c')$) are the same corners, and that consequently, $\prevF(c)=\prevV(c')$, $\prevV(c)=\prevF(c')$, $\nextE(c)=\nextE(c')$, and $\prevE(c)=\prevE(c')$.

Therefore, in the course of the opening algorithm, ran in parallel on $m$ and $\switch{\dual{m}}$, we always get to complementary states in \cref{alg:opening}. Hence the walks on the corner maps are exactly the same, and therefore the resulting maps are complement one to another.
\end{proof}

\subsection{Opening a bicolorable map}
\label{subsec:closure}

We are now willing to apply the opening algorithm to bicolorable maps with dual-geodesic orientation, and prove that this yields a bijection. We first describe some properties that will prove useful to describe the resulting maps.

\begin{defn}[well-oriented map]\label{def:wor}
A unicellular map $b$ is \emph{well-oriented} if in a tour of the face starting from the root, each edge is first followed backward and then forward. 
\end{defn}

If $b$ is a tree, this means that any interior edge is oriented toward the root.
Note that this definition does not depend on whether the tour is clockwise or counterclockwise.

Any unicellular map has a unique well-orientation, which can be straightforwardly obtained by doing a tour of the face. 
In relation to \cref{rem:unroot}, and in view of \cref{subs:wroot}, this implies that the interior orientation of a well-rooted well-oriented unicellular blossoming map $b$ can be easily recovered from the unrooted map $\undir{b}$ if we know which rootable stem of $\undir{b}$ is the root of $b$.

\begin{defn}[well-labeled map]\label{def:wlab}
A blossoming oriented map is said to be \emph{well-labeled} if its corners are labeled in such a  way that:
\begin{compactitem}
\item the labels of two corners adjacent around a vertex differ by $1$, in which case the higher label is to the right of the separating edge (or stem), 
\item the labels of two corners adjacent along an edge coincide, and
\item the root bud has labels $0$ and $1$.
\end{compactitem}
\end{defn}

Looking at the sequence of labels of corners around any fixed vertex, it is clear that the orientation of a well-labeled map is in particular Eulerian.
Note that if the map has no stem, then having a well-labeling is equivalent to having a bicolorable orientation. In particular this is stronger than having an Eulerian orientation.

The set of well-rooted well-labeled well-oriented unicellular blossoming maps, counted by vertex degrees (similarly to bicolorable maps), is denoted~$\mathcal{O}$. The subset of $\mathcal{O}$ made of $4$-valent maps is denoted $\mathcal{O}^\times$.

Recall from \cref{subs:def} that the weight of a bicolorable map $m$ is $\prod_{k>0}z_k^{v_{2k}(m)}$. Therefore, two maps have the same weight if and only if they have the same repartition of vertex degrees.

We can now state our main bijective theorem:

\begin{thm}
\label{thm:bijCut}
When performed on the dual-geodesic orientation, the opening algorithm is a weight-preserving bijection from $\mathcal{BC}_g$ to $\mathcal{O}_g$, whose inverse is the closing algorithm. 
Therefore, $BC_g(\textbf{z})=O_g(\textbf{z})$. 
\end{thm}

\begin{figure}
  \centering
  \begin{subfigure}[c]{.55\linewidth}
  \centering
   \includegraphics[scale=1,page=2]{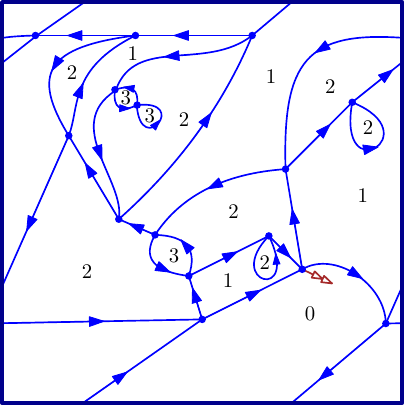}
   \caption{A map of $\mathcal{BC}_1^{\times}$ with its dual-geodesic orientation.}
  \label{fig:algosCl}
  \end{subfigure}
  \begin{subfigure}[c]{.44\linewidth}
  \centering
   \includegraphics[scale=1,page=2]{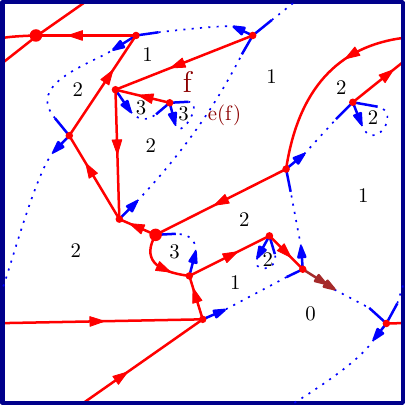}
   \caption{A map of $\mathcal{O}_1^{\times}$.}
  \label{fig:algosOp}
  \end{subfigure} \hfill
  \caption{A $4$-valent bicolorable map with dual-geodesic orientation, and its opening.}
  \label{fig:algo}
\end{figure}

Applying \cref{thm:bijCut} to $\textbf{z}=(0,z,0,\cdots)$ and using \cref{prop:radMap}, we obtain the following corollary:

\begin{cor}\label{cor:genToOp}
The opening algorithm on $4$-valent bicolorable maps yields: \[ M_g(z)=O_g^\times(z).\]
\end{cor}

\begin{rem}\label{rem:BerCha}
A more complete study of the opening algorithm in higher genus was carried in \cite{BerCha11} by Bernardi and Chapuy. 
Instead of oriented maps, they consider covered maps, that are maps with a marked unicellular spanning submap, and show that they are in correspondence with oriented maps equipped with a so-called \emph{left-connected} orientation. 
They show that the opening of such oriented maps gives rise to a unicellular spanning submap. However, this submap can be of any genus smaller or equal to the genus of the underlying surface. 
We could show that the dual-geodesic orientation is left-connected to conclude that the opening algorithm yields a unicellular spanning submap, but we would still have to prove that this map is of maximal genus. 
However, in our particular case, because the chosen orientation is not any left-connected orientation, but the dual of the geodesic orientation, we don't need to use this result. 
We show directly that the opening of a map with geodesic orientation is easily described, and use \cref{lem:dualCut} to conclude about maps with dual-geodesic orientation.
\end{rem}

\begin{proof}[Proof of \cref{thm:bijCut}]
The opening and closing algorithms are weight-preserving.
We prove in \cref{lem:BCtoO} that the image by the opening algorithm of a map of $\mathcal{BC}_g$ endowed with its dual-geodesic orientation belongs to $\mathcal{O}_g$, and in \cref{lem:OtoBC} that the image by the closing algorithm of a map of $\mathcal{O}_g$ belongs $\mathcal{BC}_g$, and is endowed with its dual-geodesic orientation.

Applying the opening algorithm to the closure of a well-rooted well-labeled well-oriented unicellular blossoming map $b$ yields the original map $b$ itself. Indeed, any closure edge is first met outgoing (see \cref{rem:wr} on well-rooted maps), whereas any non-closure edge is first met ingoing, because $b$ is well-oriented.

Reciprocally, if the opening of an oriented map $m$ of genus $g$ yields a unicellular blossoming map $b$ of genus $g$, then the closure of $b$ yields $m$. Indeed, there is a unique way to do a planar matching of the stems of $b$.
\end{proof}

\begin{lem}\label{lem:BCtoO}
Applying the opening algorithm on a bicolorable map of genus $g$ endowed with its dual-geodesic orientation yields a well-rooted well-labeled well-oriented unicellular blossoming map of genus $g$ with same weight.
\end{lem}

\begin{proof}
We look at the opening of  a bipartite map endowed with its geodesic orientation. A direct analysis of the algorithm  implies that, in this case, the blossoming map obtained is the rightmost breadth-first-search exploration tree, along with its buds and leaves.

Now let $m$ be a bicolorable map with its dual-geodesic orientation, and $o$ the opening of $m$.
Because of \cref{lem:dualCut}, we know that $\unbud{o}$ is the dual of the complement of the leftmost breadth-first-search exploration tree of $\dual{m}$ starting from the root. In particular, it is a unicellular map of maximum genus.

Since the walk on the corner map of $m$ corresponding to the opening algorithm corresponds to a clockwise tour of the unique face of $o$ starting from the root, the rules of the algorithm naturally imply that $o$ is both well-oriented and well-rooted.
If we label each corner of $o$ with the distance in $\dual{m}$ from its adjacent face to the root face, then $o$ is also well-labeled.
\end{proof}

\begin{lem}
\label{lem:OtoBC}
The closure of a map $o\in\mathcal{O}_g$ yields a bicolorable map $m$ of genus $g$ with same weight and with dual-geodesic orientation.
\end{lem}


\begin{proof}
Let $o$ be a map of $\mathcal{O}_g$ and $m=Close(o)$.
By construction, during the closing algorithm, no stem remains unmatched, and the created edges are non-crossing. This implies that $m$ is indeed correctly embedded, and has genus $g$ (and not more). 

Since $o$ is well-labeled, the height in the contour word corresponds to the labels of the corners. By consequence, the labels of corners that become adjacent along an edge by the merge of two stems are the same. Therefore, after the closure, each face of $m$ can be naturally labeled by the common label of its corners.

This labeling on faces corresponds to a labeling of dual vertices such that two adjacent vertices have label that differ by $1$ exactly. This implies that the orientation of $\dual{m}$ is a bipartite orientation, and equivalently that the orientation of $m$ is bicolorable. Hence, thanks to \cref{cor:lat}, in order to conclude that $m$ is endowed with its dual-geodesic orientation, we prove the following claim:

\begin{claim}
\label{clm:noCW}
The map $m$ has no clockwise face other than the root face.
\end{claim}
A non-root face $f$ of $m$ with label $l$ is enclosed by a certain number (possibly $0$) of edges of $o$, a certain number (possibly $0$) of edges formed by merging a bud and a leaf with adjacent labels $l$ and $l+1$, and exactly one edge formed by merging a bud and a leaf with adjacent labels $l-1$ and $l$ (see \cref{fig:exNoCW}). This edge is denoted $e(f)$.

By definition, $e(f)$ is formed by the merging of a bud and a leaf, coming in this order in a clockwise tour of the face starting from the root. By consequence, $f$ is on the left of $e(f)$, which implies that $f$ is not clockwise.

\begin{figure}  
  \centering
   \includegraphics[width=.9\linewidth]{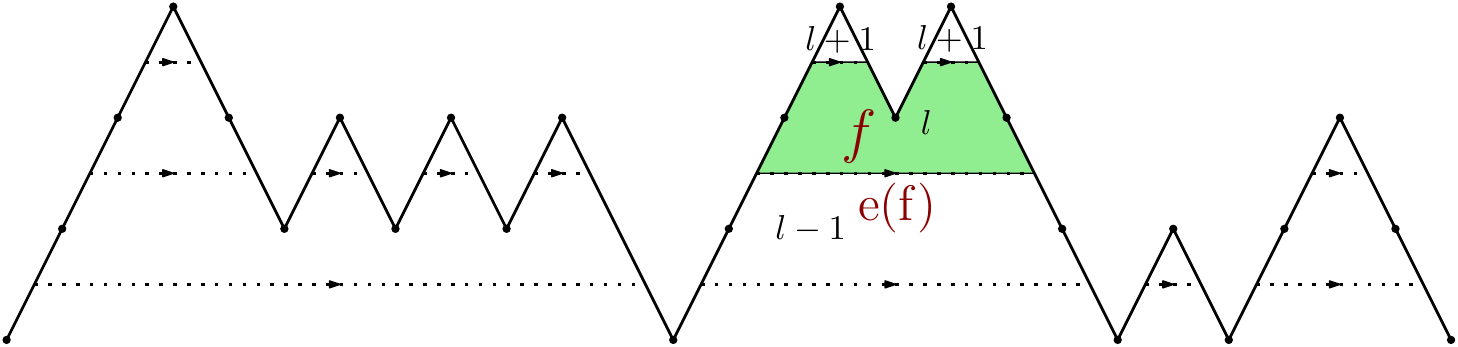}
  \caption{The contour word of the map displayed in \cref{fig:algosOp}. The face $f$, after closure, has $1$ counterclockwise adjacent edge, called $e(f)$.}
  \label{fig:exNoCW}
\end{figure}

\end{proof}

\section{Enumeration and rationality}
\label{sec:analysis}

Although the opening bijection works for any bicolorable map, we now restrict our work to $4$-valent bicolorable maps, keeping in mind that these are in bijection with general maps. The next two sections develop the analysis of the family $\mathcal{O}_g^\times$, so as to obtain a bijective proof of \cref{thm:BenCan91}, through \cref{cor:genToOp}.

\cref{tab:recap} gives a recap of the definitions and relations between some sets of maps, that have already been defined or will be in the upcoming section. It can be used as an outliner of our work up to \cref{subs:analyseAScheme}.

\begin{figure}
\centering
\begin{tabularx}{16cm}{|c|p{7.5cm}|X|}
    \hline
    $\mathcal{M}_g$ & maps of genus $g$ &  \\
    \hline
    $\mathcal{BC}^\times_g$ & $4$-valent bicolorable maps of genus $g$ & Prop.\ref{prop:radMap}: $M_g(z)=BC^\times_g(z)$ \\
    \hline
    \multirow{2}{*}{$\mathcal{O}^\times_g$} & well-rooted well-labeled well-oriented blossoming unicellular maps of genus $g$ & \multirow{2}{*}{Cor.\ref{cor:genToOp}: $M_g(z)=O^\times_g(z)$ }\\
    \hline
    \multirow{2}{*}{$\mathcal{U}_g$} & well-labeled well-oriented blossoming unicellular maps of genus $g$ & \multirow{2}{*}{Thm.\ref{thm:wroot}: $O^\times_g(z)=z^{2g-1}\cdot\frac{2}{z}\int_0^z U_g(t)dt$} \\
    \hline
    \multirow{2}{*}{$\mathcal{T}$} & rooted binary trees oriented toward the root, with $2$ buds on each inner vertex & \multirow{2}{*}{Decomposition: $T(z)=z+3T(z)^2$} \\
    \hline
    \multirow{2}{*}{$\mathcal{P}_g$} & pruned well-labeled well-oriented blossoming unicellular maps of genus $g$ & \multirow{2}{*}{Lem.\ref{lem:pruning}: $U_g(z)=\frac{\partial T}{\partial z}\cdot P_g(T(z))$} \\
    \hline
    \multirow{2}{*}{$\mathcal{R}_g$} & scheme-rooted pruned well-labeled well-oriented blossoming unicellular maps of genus $g$ & \multirow{2}{*}{Lem.\ref{lem:rerootOnScheme}: $P_g(z)\simeq\frac{1}{2g-v^s_4}\cdot\frac{\partial (tR_g(t))}{\partial t}(z)$} \\
    \hline
\end{tabularx}
\caption{A recap of some families of maps, and some relations between them.}\label{tab:recap}
\end{figure}

\subsection{Getting rid of well-rootedness}
\label{subs:wroot}

The analysis of objects such as the maps of $\mathcal{O}_g^\times$ is made difficult by the non-locality of a condition such as well-rootedness. The following theorem enables to go past that condition in the rest of the analysis.

The generating series of $\mathcal{O}_g^\times$ where maps are counted by leaves instead of vertices is denoted ${}^lO_g^\times$. A straight-forward calculation from Euler's formula gives $O_g^\times(z)=z^{2g-1}\cdot {}^lO_g^\times$.
The set of rooted (but not necessarily well-rooted) well-labeled well-oriented $4$-valent unicellular maps, counted by leaves, is denoted $\mathcal{U}$. Recall \cref{def:unroot} for the definition of the unrooted map.

\begin{thm}
\label{thm:wroot}
Let $\undir{m}$ be an unrooted map with $n+1$ rootable stems (which means its representants, the corresponding rooted maps, have $n$ leaves and $n$ buds).

There is a $(n+1)$-to-$2$ application from rooted well-labeled well-oriented $4$-valent unicellular map with unrooted map $\undir{m}$, to well-rooted well-labeled well-oriented $4$-valent unicellular map with unrooted map $\undir{m}$. 
\end{thm}

Recall the definition of a well-rootable stem, given in \cref{def:wroot}. \cref{thm:wroot} directly follows from \cref{lem:reroot}:

\begin{lem}
\label{lem:reroot}
There is a bijection, called rerooting, between rooted well-labeled well-oriented $4$-valent unicellular maps with unrooted map $\undir{m}$, $2n$ stems and a marked well-rootable stem, and well-rooted well-labeled well-oriented $4$-valent unicellular maps with unrooted map $\undir{m}$, $2n$ stems and a marked rootable stem.
\end{lem}

\cref{lem:reroot} is illustrated in \cref{fig:reroot}.

\begin{figure}  
  \begin{minipage}[c]{.48\linewidth}
   \includegraphics[width=.9\linewidth,page=2]{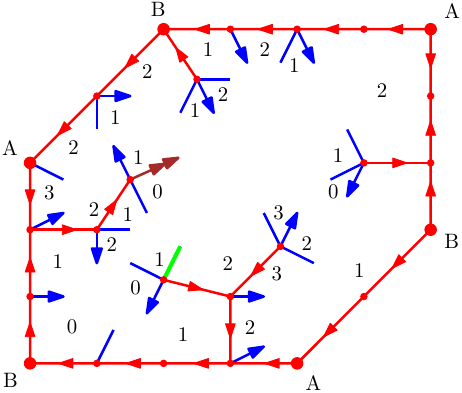}
  \end{minipage} \hfill
  \begin{minipage}[c]{.48\linewidth}
   \includegraphics[width=.9\linewidth,page=5]{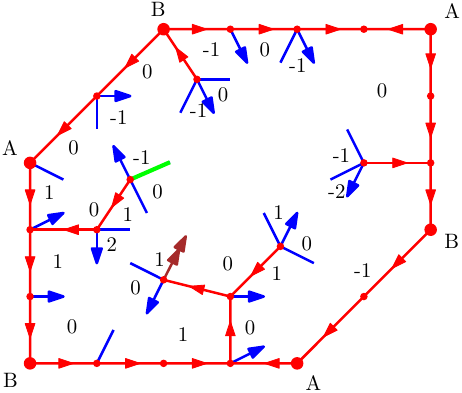}
  \end{minipage} 
 \caption{A map of $\mathcal{O}$ with a marked rootable stem (in green) is bijectively mapped (by the rerooting algorithm) to a map of $\mathcal{U}$ with a marked well-rootable stem (in green). In the $2$ maps the opposite sides are identified, so that the maps are of genus $1$, and the $2$ scheme vertices are $A$ and $B$.}
 \label{fig:reroot}
\end{figure}

\begin{proof}
Let $o$ be a well-rooted well-labeled well-oriented $4$-valent unicellular map with unrooted map $\undir{m}$, $2n$ stems and a marked rootable stem. 
 
The \emph{rerooting} algorithm is defined as follows: if the marked stem is the root, we do nothing at all. Otherwise, the root bud and the marked leaf are joined into a single oriented edge. This divides the face into $2$ faces called: $f_L$ and $f_R$, on the left and right of the newly created edge. We reduce all labels of corners of the sub-face $f_L$ by $2$. The orientation of the new edge is reversed, and it is then cut back into a bud and a leaf. The former root is marked, and the former marked leaf becomes the root bud. The interior orientation is then redefined so that the map is well-oriented, which can be easily done by doing a tour of the face.

The rerooted map is denoted $u$. It is by definition a rooted well-oriented $4$-valent unicellular map with unrooted map $\undir{m}$, $2n$ stems and a marked rootable stem. The contour word of $u$ is obtained from that of $o$ by a cyclic permutation, and by consequence, since $o$ is well-rooted, the marked edge of $u$ is well-rootable.

In order for the obtained labeling and orientation to fulfill the last condition of a well-labeled map (recall \cref{def:wlab}), in case the root does not already have labels $01$, all labels will be shifted accordingly in the end. 
However this does not alter the first $2$ conditions, and we therefore proceed to prove, before shifting, that they are satisfied.

Two labels adjacent along an edge were either both unchanged, or both reduced by $2$. 
Two labels separated by a stem which is not marked nor the root are unchanged.
Out of the two labels separated by a marked stem or the root, one is unchanged, and the other is reduced by $2$. However the orientation of the stem is also changed.
In all these cases, the labels remain compatible with the orientation after rerooting.

\begin{adjustwidth}{1.1cm}{0.5cm}
\begin{claim}
An interior edge has opposite orientation before and after rerooting if and only if it separates $f_L$ and $f_R$, in which case, before rerooting, $f_L$ is on its right and $f_R$ on its left.
\end{claim}
\begin{proof}
Because we deal with well-oriented maps, the orientation of the interior edges in the map before and after rerooting are determined by the order of apparition of their sides in a tour of the face starting from the root.

If both sides of the edge are adjacent to the same sub-face, they appear in the same order in a clockwise tour of the face starting from the root before and after rerooting, which implies that the orientation of the edge is unchanged by the rerooting.

If the two sides of the edge are not adjacent to the same sub-face, then the well-orientedness of the map before and after rerooting implies that the edge is counter-clockwise around $f_L$ and clockwise around $f_R$ before rerooting, whereas it is clockwise around $f_L$ and counter-clockwise around $f_R$ after rerooting. 
\end{proof}
\end{adjustwidth}

Now we consider two adjacent corners separated by an edge $e$, and check that their labels are compatible with the orientation of $e$. 

If the two sides of $e$ are both adjacent to $f_R$, the labels and orientation were unchanged, so they remain compatible. If the two sides of $e$ are both adjacent to $f_L$, the orientation was unchanged whereas the label were both reduced by $2$, so they remain compatible.

If $e$ has one side on each sub-face, then before rerooting, the label in $f_L$ was higher than the other one by $1$, whereas after rerooting it is reduced by $2$, and is hence smaller than the other (unchanged) label by $1$. Since both $o$ and $u$ are well-oriented, this is compatible with the change of orientation of the edge.

A very similar proof can be made for the inverse bijection.
\end{proof}

The considered families of maps can be restricted by unrooted map, so that for instance ${}^l\mathcal{O}_g^{\times \undir{m}}$ is the subset of ${}^l\mathcal{O}$ whose unrooted map is $\undir{m}$, where $m$ has to be $4$-valent and has genus $g$.

\begin{cor}\label{cor:wroot}
\cref{thm:wroot} yields: \[ {}^lO_{\undir{m}}^{\times}(t)=\frac{2}{t}\int\limits_{0}^{t}U_{\undir{m}}(z)dz. \]
\end{cor}

\subsection{Reducing a unicellular map to a labeled scheme}

The framework applied in this subsection has become classical when studying unicellular maps. In particular, it is developed by Chapuy, Marcus and Schaeffer in \cite{ChMaSc09}.

\begin{defn}[extended scheme]
The \emph{extended scheme} of a unicellular blossoming map $u$ is the unicellular map of genus $g$ obtained by iteratively removing from the interior map $\unbud{u}$ all vertices of interior degree $1$.
\end{defn}

A unicellular map $u$ is composed of an extended scheme upon which are attached some stems and treelike parts. These treelike parts, with their leaves, are binary trees, oriented towards the root of the map. Furthermore, on each interior vertex of these trees is attached a bud. The set of such trees, counted by leaves, is denoted $\mathcal{T}$. Its generating series satisfies the recurrence relation $T(z)=z+3T(z)^2$. The generating series of such trees with a marked leaf (or equivalently doubly rooted) is $z\cdot\frac{\partial T}{\partial z}(z)$.

\medskip

The \emph{pruning} procedure is defined as follows: each treelike part is replaced by a rootable stem: a root bud if the tree contains the root, a leaf otherwise (see \cref{fig:prunReroot} left and middle). 
The image of $\mathcal{U}$ by the pruning procedure, counted by leaves, is denoted $\mathcal{P}$.

\begin{lem}
\label{lem:pruning}
The pruning algorithm yields: \[ U(z)=\frac{\partial T}{\partial z}\cdot P(T(z)).\]
\end{lem}

\begin{proof}
In order to recover a map of $\mathcal{U}$ from a pruned map $p$, we need to replace each leaf of $p$ by a tree. Then the root bud of $p$ (which has weight $0$) is replaced by a tree with a marked leaf. The marked leaf is replaced by a root bud (decreasing the weight by $1$), and the tree is oriented toward this new root. The equation follows.
\end{proof}

\begin{figure}  
  \begin{minipage}[c]{.32\linewidth}
   \includegraphics[width=.9\linewidth,page=6]{longEx4UnicelMap.pdf}
  \end{minipage} \hfill
  \begin{minipage}[c]{.32\linewidth}
   \includegraphics[width=.9\linewidth,page=2]{longEx5PrunedMap.pdf}
  \end{minipage} \hfill
  \begin{minipage}[c]{.32\linewidth}
   \includegraphics[width=.9\linewidth]{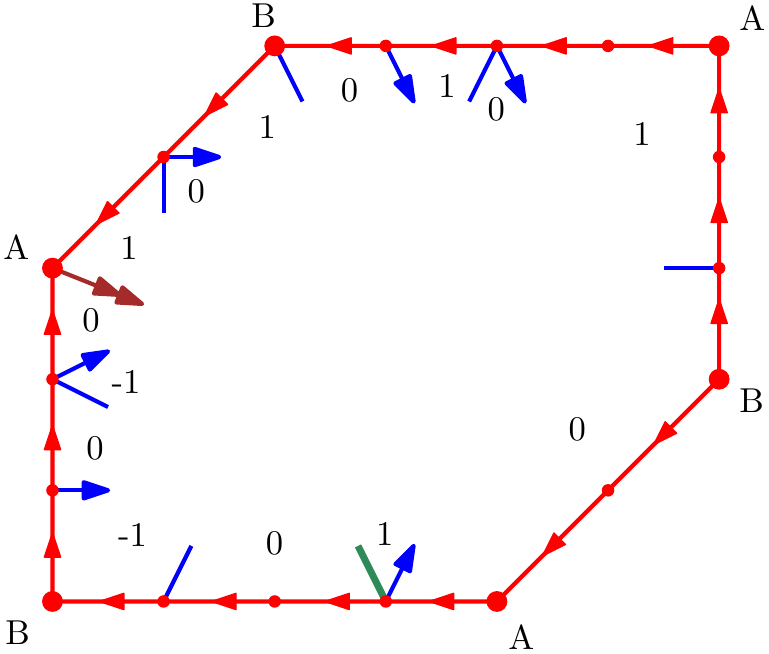}
  \end{minipage}
 \caption{On these maps, the three pairs of opposite sides are identified, so that they only have $2$ scheme vertices each, denoted $A$ and $B$. A map of $\mathcal{U}$ (left), whose treelike parts are encompassed in green, is pruned (middle). One of its scheme rootable stems is marked (in green), and the map is then rerooted (right) on this marked stem, while marking (in green) the former root stem.}
 \label{fig:prunReroot}
\end{figure}

All vertices of the pruned map are of interior degree $2$, $3$ or $4$. We call $v^s_2$, $v^s_3$, and $v^s_4$ the number of such vertices. 
When the notation is ambiguous, we specify which map is concerned by writing $v^s_2(m)$ for example.
A quick calculation based on Euler formula gives: ${v^s_3+2v^s_4=4g-2}$. There are thus a bounded number of vertices of degree $3$ or $4$, the other ones being of degree $2$. 
Vertices of interior degree at least $3$ in the pruned map are called \emph{scheme vertices}, and a stem (resp.~bud, leaf) attached on a scheme vertex (which is then necessarily of interior degree $3$ since the map is $4$-valent) is called a \emph{scheme stem} (resp. \emph{scheme bud}, \emph{scheme leaf}).
After pruning, a sequence of adjacent vertices (of interior degree $2$) between two scheme vertices is called a \emph{branch}. 

\begin{lem}
Let $p\in\mathcal{P}$. 
Out of its $v^s_3=4g-2v^s_4-2$ scheme stems, $p$ has exactly $2g-v^s_4$ rootable scheme stems. In particular $v^s_3>0$.
\end{lem}

\begin{proof}
\begin{itemize}
\item Suppose the map $p$ is scheme-rooted. Since $p$ is well-oriented, all edges of a branch are oriented the same way, which implies that all vertices of interior degree $2$ have interior out- and in-degree equal to $1$. 

By consequence, the sums of interior in- and out-degrees of scheme vertices are equal. Since the map is Eulerian, the sum of blossoming in- and out-degrees of scheme vertices are equal. Hence there are as many scheme buds as scheme leaves, that is $2g-v^s_4-1$ each.

\item Conversely, if the root is on a vertex of interior degree $2$, the root-vertex has interior in-degree $2$ and interior out-degree $0$, whereas all other vertices of interior degree $2$ have interior out- and in-degree equal to $1$. 

By consequence, the sum of interior out-degrees of scheme vertices is equal to the sum of interior in-degrees of scheme vertices, plus $2$. Since the map is Eulerian, the sum of blossoming in-degrees of scheme vertices is equal to the sum of blossoming out-degrees of scheme vertices, plus $2$. Hence there are $2g-v^s_4$ scheme leaves and $2g-v^s_4-2$ scheme buds. 
\end{itemize}

In any case there is a positive number of scheme stems, which implies that $v^s_3>0$.
\end{proof}

We now proceed to reroot the pruned map on a scheme stem. We choose a rootable scheme stem among the $2g-v^s_4$ possible choices and mark it. 
The \emph{rerooting-on-the-scheme} algorithm (see \cref{fig:prunReroot} middle and right), is the same as the rerooting described in the proof of \cref{lem:reroot}. 

The subset of $\mathcal{P}$ composed of \emph{scheme-rooted} maps is denoted $\mathcal{R}$.
We call $\mathcal{P}^{\undir{e}}$ (resp. $\mathcal{R}^{\undir{e}}$) the subset of maps of $\mathcal{P}$ (resp. $\mathcal{R}$) that have $\undir{e}$ as an unrooted extended scheme. 

\begin{lem}\label{lem:rerootOnScheme}
The rerooting-on-the-scheme algorithm yields: \[ P_{\undir{e}}(z)=\frac{1}{2g-v^s_4(\undir{e})}\cdot\frac{\partial (tR_{\undir{e}}(t))}{\partial t}(z).\]
\end{lem}

Now that the map is rooted on a scheme bud, since it is well-oriented, all edges of a branch have the same orientation. We call \emph{merging} the procedure that replaces each branch by a single edge with the same orientation (see \cref{fig:redToScheme}). 

The map we obtain is called the \emph{labeled scheme}. It is not well-labeled because corners adjacent along an edge do not necessarily have the same label anymore, but the rule around a vertex is still respected. The set of labeled schemes is denoted $\mathcal{L}$.

\subsection{Analyzing a scheme}\label{subs:analyseAScheme}
For $l\in\mathcal{L}$, we now want to determine which maps have $l$ as labeled scheme. 
Each edge of $l$ should be replaced by a valid branch. However we need to be sure that after replacement, the map is well-labeled, and agrees with the labeling of the scheme. Therefore, following \cite{ChMaSc09}, we express the generating series of branches with prescribed height on the extremities.

\begin{figure}
\includegraphics[width=.95\textwidth]{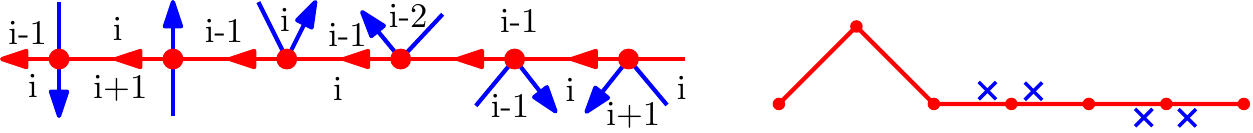}
\caption{An example of a branch displaying the $6$ possible types of vertices of degree $2$, and the corresponding weighted Motzkin path.}
\label{fig:deg2vertex}
\end{figure}

There are $6$ types of vertices of interior degree~$2$, displayed in \cref{fig:deg2vertex} left. If the bud and leaf are on opposite sides, the label of the corners either increases on both side or decreases on both sides. In the $4$ other cases, the stems are on the same side, and the label remains the same before and after the vertex. Therefore each type of vertex of interior degree $2$ can be represented by a step, depending on the variation of the labels around it: an up-step if the label increases, a down-step if it decreases, and 4 types of horizontal steps if it stays the same, represented with a blue cross placed accordingly to the position of the bud (see \cref{fig:deg2vertex} right). These steps are called \emph{weighted Motzkin steps}, and together they form a \emph{weighted Motzkin path}, whose variation of height corresponds to the variation of labels of the corresponding branch.

An edge of the labeled scheme going from label $i$ to label $j$ can therefore be replaced by a weighted Motzkin path going from height $i$ to height $j$, as illustrated in \cref{fig:toLabSch}. 

\begin{figure}  
\label{fig:toLabSch}
  \centering
   \includegraphics[width=.9\linewidth,page=11]{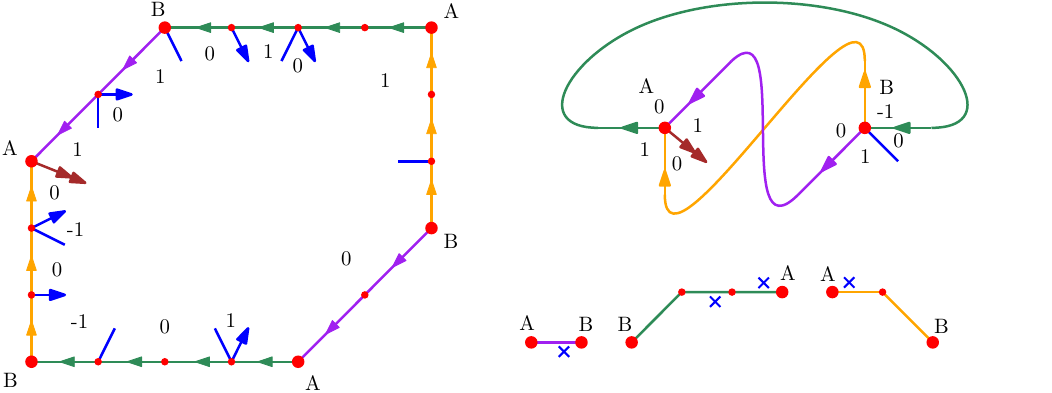}
  \caption{Reducing a map of $\mathcal{R}$ to a labeled scheme, by replacing each branch by a weighted Motzkin path.}
  \label{fig:redToScheme}
\end{figure}

We denote by $\mathcal{D}$ the set of weighted Motzkin paths going from height $0$ to height $-1$, that remain non-negative before the last step, counted by length. It satisfies the decomposition equation: $D=z(1+4D+D^2)$.
We denote by $\mathcal{B}$ the set of weighted Motzkin paths going from height $0$ to height $0$, counted by length. It satisfies the decomposition equation: $B=1+4zB+2zDB$. 

After combination with the previous equation, this equation is rewritten as a function of $D$ only: ${B=\frac{1+4D+D^2}{1-D^2}}$.
The generating series of paths going from height $i$ to $j$ is: $B\cdot D^{|i-j|}$. 

\begin{rem}
The role of $B$ and $D$ is very similar to the role of $B$ and $U$ in the work of Chapuy, Marcus and Schaeffer in \cite{ChMaSc09}. A few subtle difference may be noted: in our case there are $4$ different horizontal steps instead of only $1$. Furthermore, an element of $\mathcal{B}$ may be of length $0$, which is not the case in \cite{ChMaSc09}, and leads to simpler formulae.
\end{rem}

Recall that our purpose is to prove that $M_g(t)$, the series of maps of genus $g$, is rational in $t$. Using the lemmas of \cref{sec:analysis}, we will be able to express the generating series in terms of the auxiliary function $D$. A key observation of \cite{ChMaSc09} is that rationality in $t$ amounts to symmetry in $D$, thanks to \cref{lem:ratSym}.

For a function $\Psi$ in $D$, we denote its \emph{transposition} $\transpose{\Psi}(D)=\Psi(D^{-1})$. We say that $\Psi$ is \emph{symmetric} (resp. \emph{antisymmetric}) in $D$ if $\Psi=\transpose{\Psi}$ (resp. $\Psi=-\transpose{\Psi}$). Note for example that $B$ is antisymmetric.

\begin{lem}\label{lem:ratSym}
A function is rational in $z$ if and only if is it is rational and symmetric in $D$.
\end{lem}

\begin{proof}
Since $z=\frac{1}{D^{-1}+4+D}$, any function which is rational in $z$ is rational and symmetric in $D$. 

Let $f$ be a function rational and symmetric in $D$, whose irreducible expression is $\frac{P}{Q}$. We denote $d$ the average degree of $P$, which is half the sum of the higher and lower degree of $P$. By symmetry, it is also the average degree of $Q$. By symmetry, $P\cdot D^{-d}$ and $Q\cdot D^{-d}$ are both symmetric in $D$ and $D^{-1}$. In case $d$ is not an integer, $P\cdot D^{-d}$ and $Q\cdot D^{-d}$ are not polynomials in $D$ and $D^{-1}$. However, in this case, $P\cdot D^{-d}\cdot(D^{-\frac{1}{2}}+D^{\frac{1}{2}})$ and $Q\cdot D^{-d}\cdot(D^{-\frac{1}{2}}+D^{\frac{1}{2}})$ are symmetric polynomials in $D$ and $D^{-1}$.

Therefore, in any case, $f$ can be written as the ratio of $2$ symmetric polynomials. Since the family of polynomials $z^{-k}=(D^{-1}+4+D)^k$ generates all symmetric polynomials, $f$ can be written as a rational function of $z^{-1}$, which is enough to conclude the proof.
\end{proof}

\medskip

An \emph{unlabeled scheme} is a scheme where we forgot all labels. We denote by $\mathcal{S}$ the set of unlabeled schemes. The set of unrooted unlabeled scheme is denoted $\overline{\mathcal{S}}$.
We specialize our classes of maps depending on their unlabeled or unrooted scheme, by writing $\mathcal{M}_s$ or $\mathcal{M}_{\overline{s}}$ for example, where $s$ and $\overline{s}$ are an unlabeled scheme and an unrooted scheme. 

We denote $R^b$ the generating series of $\mathcal{R}$, counted only by leaves attached on a branch, instead of all leaves, and $R^b_s$ its restriction to maps with unlabeled scheme $s$. 
Note that \(R_s(z)=z^{2g-v_4^s(s)-1}R^b_s(z)\).

\begin{thm}
\label{thm:sym}
For any $s\in\mathcal{S}$, $R^b_{s}$ is rational and symmetric in $D$.
\end{thm}

\cref{sec:symmetry} will be dedicated to prove \cref{thm:sym}. 
To that end, an additional study on the structure of the schemes will be required, that will be carried on in \cref{subs:offset}.

In the rest of \cref{subs:analyseAScheme}, we show that \cref{thm:sym}, in conjunction with the work of \cref{sec:analysis}, leads to \cref{thm:ratByScheme}, a refined version of \cref{thm:BenCan91}, where the unrooted scheme obtained by our bijection is specified.

Why do we restrict classes of maps by unrooted scheme rather than by unlabeled scheme?
Recall from \cref{def:unroot} that an unrooted scheme is an unlabeled and undirected scheme where we forgot which rootable stem is the actual root. Forgetting this information is useful, because the rerooting procedures are many-to-many applications rather than bijections, which means that going through our bijection does not associate a fixed unlabeled scheme to a given map; however, all unlabeled schemes associated to a map have the same unrooted scheme.

\begin{thm}\label{thm:ratByScheme} For any $\overline{s}$ in $\overline{\mathcal{S}}$, the generating series $M_{\overline{s}}(t)$ is a rational function of $T(t)$. 
\end{thm}
 
Since $\overline{\mathcal{S}}_g$ is finite for any fixed $g$, this theorem implies that $M_g(t)=\sum_{\overline{s}\in \overline{\mathcal{S}}_g}M_{\overline{s}}(t)$ is rational in $T(t)$, which is equivalent to \cref{thm:BenCan91}.

\begin{rem}
The main reason why we are able to obtain the rationality of $M_g$ is that, unlike in \cite{ChMaSc09}, the rationality holds for each scheme, which makes it possible to analyse more specifically one scheme at a time. The reason why maps are rational by scheme in our case but not in \cite{ChMaSc09} remains somewhat of a mystery.
\end{rem}

\begin{proof}[Proof of \cref{thm:ratByScheme}]
We derive from the previous sections (see also \cref{tab:recap}) that:
\begin{equation}
\begin{array}{rll}
M_{\overline{s}}(t)&=BC^{\times}_{\overline{s}}(t) \\
&=O^{\times}_{\overline{s}}(t),
&\text{ because of \cref{cor:genToOp}}\\
&=t^{2g-1}\cdot {}^lO_{\overline{s}}(t),
&\text{ by definition}\\
&=t^{2g-1}\cdot\frac{2}{t}\int_0^t U_{\overline{s}}(z)dz,
&\text{ because of \cref{thm:wroot}}\\
&=2t^{2g-2}\int_0^t \frac{dT}{dz}\cdot P_{\overline{s}}(T(z))\cdot dz,
&\text{ because of \cref{lem:pruning}}\\
&=\frac{2t^{2g-2}}{2g-v^s_4(s)}\int_0^t \frac{d(uR_{\overline{s}}(u))}{du}(T(z))\cdot\frac{dT}{dz} \cdot dz,
&\text{ because of \cref{lem:rerootOnScheme}}\\
&=\frac{2t^{2g-2}}{2g-v^s_4(s)}\cdot T(t)\cdot R_{\overline{s}}(T(t)),
&\text{ by a change of variable}\\
\end{array}
\end{equation}

Hence, in order to prove that $M_{\overline{s}}(t)$ is rational in $T$ (and $t$), it suffices to prove that $R_{\overline{s}}(z)$ is rational in $z$, or equivalently that $R_{\overline{s}}$ is rational and symmetric in $D$, thanks to \cref{lem:ratSym}.
For all $\overline{s}\in\overline{\mathcal{S}}$, we have $R_{\overline{s}}=\sum_{\stackrel{t\in\mathcal{S}}{\overline{t}=\overline{s}}}R_{t}$. 

Therefore, \cref{thm:sym} is enough to conclude the proof.
\end{proof}

\begin{rem}\label{rem:shortcut}
Note that we apply twice the rerooting algorithm, the first time from a well-rootable stem to any rootable stem, and the second time from a rootable stem to a rootable scheme stem. In the course of the proof of \cref{thm:ratByScheme}, these two operations, in terms of generating functions, correspond to an integral and a derivative. Although the two rerooting operations are separated by the pruning operation, it appears that a change of variable allows the integral and derivative to cancel out. This can actually be seen in a combinatorial way, by merging the two rerooting operations and the pruning operation into a single operation, which result in \cref{lem:shortcut} 
\end{rem}

\begin{lem}\label{lem:shortcut}
For any unrooted scheme $\overline{s}$, there is a $(2g-v^s_4(\overline{s}))$-to-$2$ application from maps of $\mathcal{O}_{\overline{s}}^{\times}$ to maps of $\mathcal{R}_{\overline{s}}$ with a tree associated to each of its rootable stems.
\end{lem}

\subsection{The offset graph}\label{subs:offset}

We now label each scheme vertex with the minimal label of its corners, and relabel its corners relatively to this minimum. This second label is called the \emph{offset label}. In case the outgoing and ingoing edges around a vertex are alternated, the offset labels around the vertex are $0101$. Otherwise, the sequence is $0121$.

The edges of the scheme can be of two different types. Look at the offset labels around it. If the offset labels are the same ($01$ or $12$) on both sides, the edge is called \emph{level}. If the labels are $01$ on one side and $12$ on the other, the edge is \emph{offset} toward the second one.
We define the \emph{offset graph} as the oriented sub-graph of the scheme where only the offset edges are kept, along with their orientation. 
See \cref{fig:exOffset} for an example.

\begin{figure}  
  \centering
  \begin{subfigure}[c]{.45\linewidth}
  \centering
   \includegraphics[width=.9\linewidth,page=4]{longEx7EmbeddedScheme}
  \end{subfigure}
  \begin{subfigure}[c]{.45\linewidth}
  \centering
   \includegraphics[width=.9\linewidth,page=1]{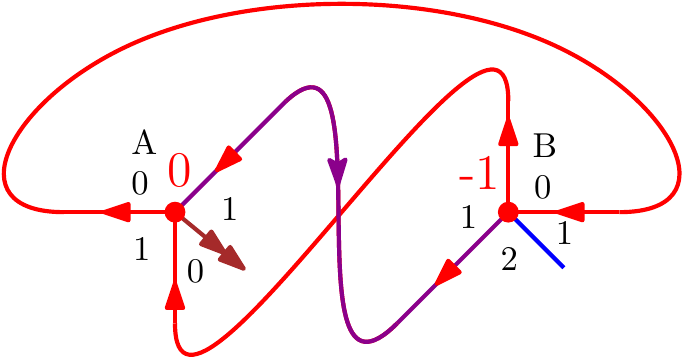}
  \end{subfigure} 
  \caption{The scheme on the left has the usual labels. On the right, these labels define labels on vertices (red), and offset labels (black). An offset oriented edge (purple) appears. Here, the offset graph is reduced to this single offset edge.}
  \label{fig:exOffset}
\end{figure}

\begin{prop}
\label{prop:acyclic}
The offset graph of a scheme is acyclic.
\end{prop}

\begin{proof}
Let's assume by contradiction that there exists a labeled scheme $l\in\mathcal{L}$ whose offset graph is not acyclic and let $C$ be a simple cycle of the offset graph of $l$. Since a $0101$-type vertex can only be a source in the offset graph, all vertices of $C$ are of type $0121$.

Let $e_1$ be the first edge of $C$ visited during a clockwise tour of $l$ starting from the root. Since $l$ is well-oriented, $e_1$ is visited backward first. Depending on whether it is visited forward or backward in the offset graph, we are in one of the two cases depicted in \cref{fig:acycOff}.

\begin{figure}
\centering
\includegraphics[scale=1.2]{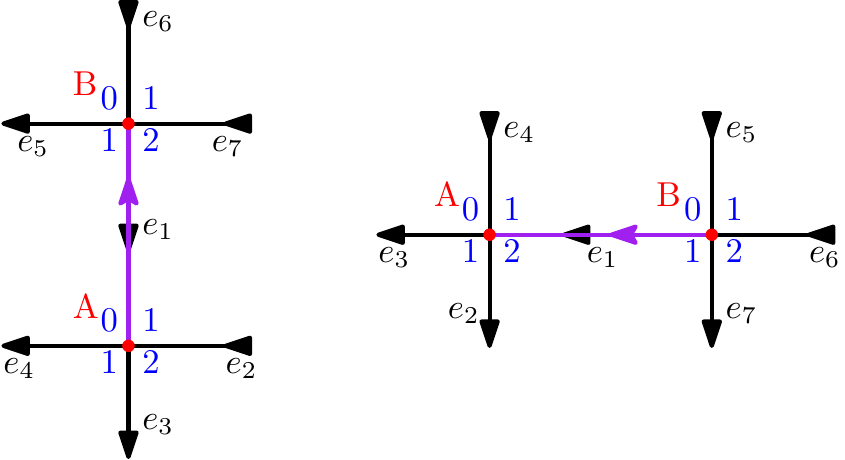}
\caption{illustration of the two possible cases of first visited edge of an offset cycle. In both cases, $e_2$ to $e_7$ can \textit{a priori} be either stems or edges, and belong or not to the offset graph.}
\label{fig:acycOff}
\end{figure}

\begin{itemize}
\item Assume we are in the leftmost case. Consider now the first visited edge (or stem) adjacent to vertex $A$, and call it $e_f(A)$. Be it the root bud or a normal edge, $e_f(A)$ has to be visited backward first in the tour, and hence it is either $e_3$ or $e_4$. 

Suppose $e_f(A)=e_4$, then $e_3$ cannot be an interior edge because it would have to be visited backward first in the tour, whereas it is outgoing from $A$ and cannot have been visited before $e_4$ by definition of $e_f(A)$. Hence $e_3$ is a bud, and $e_2$ is visited before $e_1$ in the tour. 
If $e_f(A)=e_3$, then $e_2$ is visited before $e_1$. 

In any case both $e_2$ and $e_3$ are visited before $e_1$, which implies they are not part of $C$. This is a contradiction because $C$ is a cycle which is oriented in the offset graph, which means it does not contains $e_4$, because $e_4$ has relative label $01$ around $A$.

\item Now assume we are in the rightmost case. We define similarly $e_f(B)$ to be the first visited edge (or stem) around $B$. Again, since it has to be outgoing, it is either $e_1$ or $e_7$.

Suppose $e_f(B)=e_7$. Then both $e_7$ and $e_6$ are visited before $e_1$ in the tour, which implies they are not in $C$. Moreover, since $e_5$ has relative label $01$ around $B$, it cannot be part of the offset cycle $C$. This is a contradiction.

Therefore $e_f(B)=e_1$. Then similarly to the leftmost case, we conclude that $e_7$ is a bud. By consequence, since $C$ is a cycle in the offset graph, the edge coming right after $e_1$ in $C$ has to be $e_6$, which has relative label $12$ around $B$.
\end{itemize}

The same reasoning (we already know we are in the rightmost case) can be applied on $e_6$, the second visited vertex of $C$, and so on. By induction we conclude that $C$ is only made of edges whose orientation in the map and the offset graph coincide. Additionally, the only edges or stems outgoing from $C$ are on its left, whereas the only edges or stems ingoing from $C$ are on its right. We also know that the left part of $C$ is visited backward all at once, from $e_1$ to $e_3$.

Now, in a clockwise tour of the face, which edge (or stem) $e_f^r(C)$ is visited just before the first time we meet the right side of $C$? Since in the tour, $e_f^r(C)$ has to be visited backward first (even if it is the root), it can not be any of the edges or stems on the right of $C$.
Since there is no good candidate for $e_f^r(C)$, there is a contradiction, which implies that the offset graph is acyclic.
\end{proof}

\section{Rationality of maps with a given scheme}
\label{sec:symmetry}

In this section we prove \cref{thm:sym}. Throughout all the section, $s$ is an unlabeled scheme.

\subsection{Counting rerooted pruned maps}

We now proceed to compute the generating function of all maps with a fixed unlabeled scheme. A branch from height $i$ to height $j$ in a scheme contributes for $B\cdot D^{|i-j|}$. To get the total participation of an unlabeled scheme to $R^b_s$, we need to sum over all possible labelings of the vertices. To deal with the relativity of label, instead of requiring that the root has label $01$, we arbitrarily decide that the lowest height has to be $0$, which is not equivalent since the maps are not necessarily well-rooted.

To make things simpler, we first analyse the case of an unlabeled scheme $s$ with no offset edge. In this case, we obtain:

\begin{equation}
R^b_s=
\sum\limits_{\stackrel
{h_1\cdots h_{n_v}\in\mathbb{N}}
{\text{min}(h_1,\cdots, h_{n_v})=0}}
\prod\limits_{\stackrel{(v_i,v_j)\in\mathcal{E}(s)}{i<j}}
B\cdot D^{|h_i-h_j|}.
\end{equation}

For a given affectation of heights, we define a function characterizing the ordering of heights of vertices. We set $k$ to be the total number of distinct heights. To each vertex $v_i$ we associate an integer $o(i)$ between $1$ and $k$ such that $h_i<h_j$ (resp. $h_i=h_j$, $h_i>h_j$) if and only if $o(i)<o(j)$ (resp. $o(i)=o(j)$, $o(i)>o(j)$). Note that $o$ is necessarily a surjection. The size of the image of a surjection $o$ is denoted $k(o)$. Whenever it is unambiguous, we simply write $k$. Let $[k]=\{1,2,\cdots,n\}$. We denote $S(n)$ the set of surjection from $[n]$ to $[k]$, for any $k$ in $[n]$.

We can therefore rewrite the formula for the generating series:
\begin{equation}
R^b_s=
B^{n_e}\cdot 
\sum\limits_{o\in S(n_v)}
\sum\limits_{0=h_1<\cdots<h_{k}}
\prod\limits_{\stackrel{(v_i,v_j)\in\mathcal{E}(s)}{i<j}}
D^{|h_{o(i)}-h_{o(j)}|}.
\end{equation}

\begin{rem}
This idea of grouping the labeled schemes that have the same relative ordering of scheme vertices is reminiscent to the use of \emph{standard schemes} in \cite{ChMaSc09}.
\end{rem}

For a subset $S$ of vertices, we define the cut of a subgraph as: \( C(S)=|\{(u,v)\in\mathcal{E}(s) \text{ such that } u\in S \text{ and } v\notin S \}|.\)
We also define: $\Phi_o(i)=\frac{D^{C(o^{-1}([i]))}}{1-D^{C(o^{-1}([i]))}}.$

\begin{lem}\label{lem:noOffset}
For an unlabeled scheme $s\in\mathcal{S}$ with no offset edge, we have:
\begin{equation}
R^b_s=
B^{n_e}\cdot 
\sum\limits_{o\in S(n_v)}
\prod\limits_{i=1}\limits^{k-1}
\Phi_o(i).
\end{equation}
\end{lem}

\begin{proof}

We apply the change of variables: $I_i=h_{i+1}-h_{i}$:
\begin{equation}
R^b_s=
B^{n_e}\cdot
\sum\limits_{o\in S(n_v)}
\sum\limits_{I_1,\cdots,I_{k-1}\in\mathbb{N}^*}
\prod\limits_{\stackrel{(v_i,v_j)\in\mathcal{E}(s)}{o(i)\leq o(j)}}
D^{I_{o(i)}+I_{o(i)+1}+\cdots+I_{o(j)-1}}.
\end{equation}

Our expression can then be rewritten in terms of cut:
\begin{align}
R^b_s&=
B^{n_e}\cdot 
\sum\limits_{o\in S(n_v)}
\sum\limits_{I_1,\cdots,I_{k-1}\in\mathbb{N}^*}
\prod\limits_{i=1}\limits^{k-1}
D^{I_i\cdot C(o^{-1}([i]))}\\
&=B^{n_e}\cdot 
\sum\limits_{o\in S(n_v)}
\prod\limits_{i=1}\limits^{k-1}
\sum\limits_{I_i\in\mathbb{N}^*}
D^{I_i\cdot C(o^{-1}([i]))}\\
&=
B^{n_e}\cdot 
\sum\limits_{o\in S(n_v)}
\prod\limits_{i=1}\limits^{k-1}
\frac{D^{C(o^{-1}([i]))}}{1-D^{C(o^{-1}([i]))}}.
\end{align}

\end{proof}

Now we take offset edges into account. Since the offset graph is acyclic (\cref{prop:acyclic}), we relabel the vertices so that for all oriented offset edge $(v_i,v_j)$ with $i<j$, the edge is offset toward $j$. In other words, vertices are relabeled according to a linear extension of the partial order induced by the offset graph.

Consider an offset edge $e=(v_i,v_j)\in\mathcal{O}$ with $i<j$. We say that $e$ is a \emph{tie} (resp.~an \emph{inversion}, resp.~an \emph{anti-inversion}) if $o(i)=o(j)$ (resp.~$o(i)>o(j)$, resp.~$o(i)<o(j)$). We call $n_t$ (resp.~$n_i$, resp.~$n_a$) the number of ties (resp.~inversions, resp.~anti-inversions). Whenever it is necessary, we specify which surjection is concerned by writing $n_t(o)$ for instance.

\begin{lem}\label{lem:withOffset}
For any unlabeled scheme $s\in\mathcal{S}$, we have:
\begin{equation}
R^b_s=
B^{n_e}\cdot 
\sum\limits_{o\in S(n_v)}
\left(\prod\limits_{i=1}\limits^{k-1}\Phi_o(i)\right)\cdot
D^{n_t(o)+n_a(o)-n_i(o)}.
\end{equation}
\end{lem}

\begin{proof}
Consider the formula given in \cref{lem:noOffset}. To take offset edges into account, we need to change slightly the weight of each offset edge $e$.
If $e$ is a tie, the corresponding branch was given weight $B$ instead of $B\cdot D$. If $e$ is an inversion, the corresponding branch was given weight $B\cdot D^{h_{o(i)}-h_{o(j)}}$ instead of $B\cdot D^{h_{o(i)}-h_{o(j)}-1}$. If $e$ is an anti-inversion, the corresponding branch was given weight $B\cdot D^{h_{o(j)}-h_{o(i)}}$ instead of $B\cdot D^{h_{o(j)}+1-h_{o(i)}}$. To take these defects into account we therefore need to multiply by a factor depending on the number of inversions, ties, and anti-inversions. This leads to \cref{lem:withOffset}.
\end{proof}

\subsection[Proof of the symmetry]{Proof of \cref{thm:sym}}

We start from the formula of \cref{lem:withOffset} and use the fact that $\transpose{\Phi_o(i)}=-(1+\Phi_o(i))$. Recall that $B$ is asymmetric. Therefore,

\begin{align}
\transpose{R^b_s}&=\transpose{B^{n_e}}\cdot 
\sum\limits_{o\in S(n_v)}
\left(\prod\limits_{i=1}\limits^{k-1}\transpose{\Phi_o(i)}\right)\cdot
D^{-n_t-n_a+n_i}\\
&=(-1)^{n_e}\cdot B^{n_e} \cdot 
\sum\limits_{o\in S(n_v)}
\left((-1)^{k-1}\cdot
\prod\limits_{i=1}\limits^{k-1}(1+\Phi_o(i))\right)\cdot
D^{-n_t-n_a+n_i}.
\end{align}

Given two surjections $o$ and $p$, we say that $o$ \emph{refines} $p$ and $p$ \emph{coarsens} $o$, and write $\refine{o}{p}$, if: \begin{equation}\forall x,y; \quad o(x)=o(y)\Rightarrow p(x)=p(y).\end{equation}
This means equivalently that the partial order on vertices induced by $o$ is an extension of the one induced by $p$. 

A term of the development of the product $\prod_{i=1}^{k(o)-1}(1+\Phi_o(i))$ corresponds to a permutation $p$ that coarsens~$o$; indeed, for each $i$, we have to choose $1$ or $\Phi_o(i)$ in the product, which corresponds to choosing whether or not to merge $o^{-1}(i)$ and $o^{-1}(i+1)$. Therefore, $\prod_{i=1}^{k(o)-1}(1+\Phi_o(i))=\sum_{\coarsen{p}{o}}\prod_{i=1}^{k(p)-1}\Phi_p(i).$

Hence we can rewrite the previous expression by first developing the product, and then interverting the two summations:
\begin{align}
\transpose{R^b_s}
&=(-1)^{n_e}\cdot B^{n_e} \cdot 
\sum\limits_{o\in S(n_v)}
\left((-1)^{k(o)-1}\cdot
\sum\limits_{\coarsen{p}{o}}
\prod\limits_{i=1}\limits^{k(p)-1}\Phi_p(i)\right)\cdot
D^{-n_t(o)-n_a(o)+n_i(o)}
\\
&=(-1)^{n_e}\cdot B^{n_e} \cdot 
\sum\limits_{p\in S(n_v)}
\prod\limits_{i=1}\limits^{k(p)-1}\Phi_p(i)\cdot
\sum\limits_{\refine{o}{p}}\left((-1)^{k(o)-1}\cdot
D^{-n_t(o)-n_a(o)+n_i(o)}\right).
\end{align}

The \emph{reverse} $\overline{p}$ of a surjection $p$ is defined as follows: $\overline{p}(i)=k(p)+1-p(i)$. 
Since $C(\mathcal{S})=C(\mathcal{S}^\complement)$,  $\prod_{i=1}^{k(p)-1}\Phi_p(i)=\prod_{i=1}^{k(\overline{p})-1}\Phi_{\overline{p}}(i)$. As a consequence:

\begin{equation}
\transpose{R^b_s}
=(-1)^{n_e}\cdot B^{n_e} \cdot 
\sum\limits_{p\in S(n_v)}
\prod\limits_{i=1}\limits^{k(\overline{p})-1}\Phi_{\overline{p}}(i)\cdot
\sum\limits_{\refine{o}{p}}\left((-1)^{k(o)-1}\cdot
D^{-n_t(o)-n_a(o)+n_i(o)}\right)
\end{equation}

We apply the change of variable $p\rightarrow \overline{p}$:

\begin{equation}
\transpose{R^b_s}=(-1)^{n_e}\cdot B^{n_e} \cdot 
\sum\limits_{p\in S(n_v)}
\prod\limits_{i=1}\limits^{k(p)-1}\Phi_p(i)\cdot
\sum\limits_{\refine{o}{\overline{p}}}\left((-1)^{k(o)-1}\cdot
D^{-n_t(o)-n_a(o)+n_i(o)}\right).
\end{equation}

To conclude that $R^b_s$ is symmetric, it is now enough to prove that for any offset graph, and for any surjection $p$, the following lemma holds:

\begin{lem}\label{lem:sumCancel}
For any surjection $p$:
\begin{equation}
\sum_{\refine{o}{\overline{p}}}\left((-1)^{k(o)-1}\cdot
D^{-n_t(o)-n_a(o)+n_i(o)}\right)=(-1)^{n_e}\cdot D^{n_t(p)+n_a(p)-n_i(p)}.
\end{equation}
\end{lem}

The general case of \cref{lem:sumCancel} will be proved in \cref{subs:surj}.

\begin{rem}
Note that, when the offset graph is empty, this formula can be obtained as a direct byproduct of the Euler-Poincaré formula applied to the permutahedron.
The $n$-permutahedron (see \cref{fig:permutahedra}) is a polytope defined as the convex hull of the set of permutations: 
${Perm_n=Conv\{(\sigma_1,\sigma_2 \cdots \sigma_n) \text{ such that } \sigma\in \mathfrak{S}_n\}},$ where $\mathfrak{S}_n$ is the set of permutations of size $n$. 
The $n$-permutahedron has dimension $n-1$. 
The $k$-dimensional faces of the $n$-permutahedron correspond to surjections from $[n]$ to $[n-k]$.
If $\refine{o}{p}$, then the face corresponding to $o$ is included into that of $p$.

The Euler-Poincaré formula states that, if $f_k$ denotes the number of $i$-dimensional faces of a polytope, then: $\sum_{k\geq 0}(-1)^{k}f_k=0.$ 
A face of a polytope is also a polytope; in the case of the permutahedron, it is even the Cartesian product of lower-dimensional permutahedra. 
Therefore, the Euler-Poincaré formula applied to the face corresponding to $\overline{p}$ implies: $\sum_{\refine{o}{\overline{p}}}(-1)^{k(o)-1}=(-1)^{n_e}.$
\end{rem}

\begin{figure}
  \centering
  \begin{subfigure}[c]{.46\linewidth}
   \centering
   \includegraphics{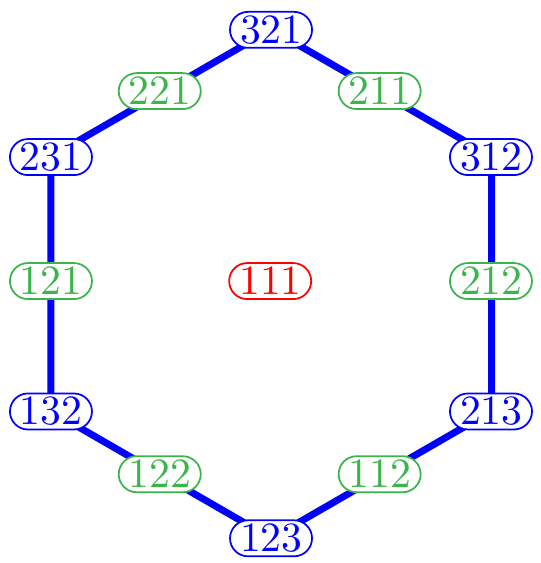}
  \end{subfigure}
  \begin{subfigure}[c]{.46\linewidth}
   \centering
   \includegraphics[scale=0.8]{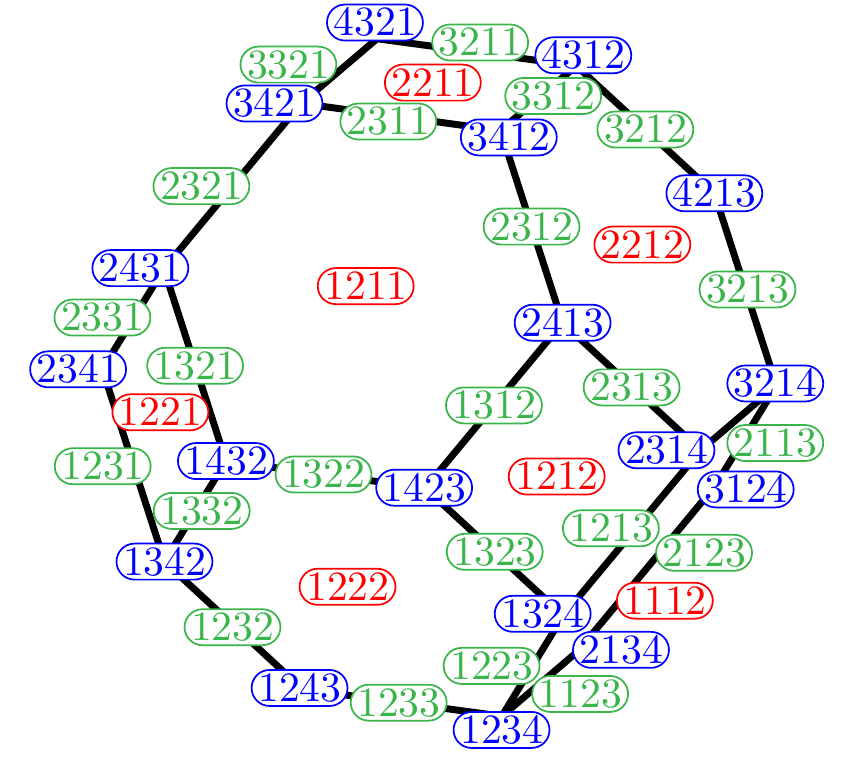}
  \end{subfigure}
  \caption{\label{fig:permutahedra}Permutahedra of dimension $2$ and $3$. Faces of dimension $0$, $1$, and $2$ are represented in blue, green, and red.}
\end{figure}

\subsection[Proof of the sum cancellation]{Proof of \cref{lem:sumCancel}}
\label{subs:surj}

We now consider a more general context than above, with a variable $X_{ij}$ for all $i<j$. For a surjection $p$ of size $n_v$ we define a monomial $X(p)=\prod_{i<j}(\delta_{p(i)>p(j)}X_{ij}+ \delta_{p(i)\leq p(j)}X_{ij}^{-1})$. 
Note that no two permutations have the same monomial.

\begin{prop}\label{prop:sumX}
The reverse of the monomial associated to an ordered partition $p$ is equal to the alternated sum of monomials associated to the ordered partitions that refine the reverse of $p$:
\begin{equation}X(p)^{-1}=\sum\limits_{\refine{o}{\overline{p}}}(-1)^{k(o)-n_v}X(o). \end{equation}
\end{prop}

\cref{prop:sumX} is a direct consequence of \cref{lem:surj1,lem:surj2}. We first use \cref{prop:sumX} to prove \cref{lem:sumCancel}, and then state and prove \cref{lem:surj1,lem:surj2}.

\begin{proof}[Proof of \cref{lem:sumCancel}]
Recall that the vertices have been relabeled in an order which is a linear extension of the offset graph, and by consequence any oriented edge $(i,j)$ of the offset graph satisfies $i<j$.

We use \cref{prop:sumX}, and specialize $X_{ij}$ to $D$ to a power equal to the number of oriented offset edges $(i,j)$. We obtain:

\begin{align}
X(p)^{-1}&=\sum\limits_{\refine{o}{\overline{p}}}(-1)^{k(o)-n_v}X(o)\\
D^{n_t(p)+n_a(p)-n_i(p)}
&=\sum\limits_{\refine{o}{\overline{p}}}(-1)^{k(o)-n_v}\cdot
D^{-n_t(o)-n_a(o)+n_i(o)}.
\end{align}

The Euler formula applied to a scheme states that $n_v-n_e=1-2g$ and by consequence $(-1)^{n_v}=(-1)^{n_e+1}$, which conclude the proof.
\end{proof}

For a given surjection $p$, we define a \emph{canonical permutation} $r(p)$ as follows: 
\[r(p)(i)=|\{j\text{ such that }p(j)<p(i)\}|+|\{j\text{ such that }p(j)=p(i)\text{ and }j\geq i\}|.
\]
This permutation is the linear extension of the partial order induced by $p$, such that each time a choice is to be made, elements are taken in decreasing order.

\begin{lem}\label{lem:surj1}
The alternated sum of monomials of surjections that refine a given surjection $p$ is equal to the monomial of the canonical permutation of $p$: \begin{equation} \sum\limits_{\refine{o}{p}}(-1)^{k(o)}X(o)=(-1)^{n_v}X(r(p)).\end{equation}
\end{lem}

\begin{proof}
We group all surjections refining $p$ that have the same monomial. For any monomial, there is exactly one such surjection that is a permutation. 

Let $q$ be a permutation that refines $p$.
An \emph{admissible ascent} of a permutation $q$ is a pair of numbers $(i,i+1)$ such that $q(i+1)<q(i)$ and $p(i+1)=p(i)$. 
The surjections that refine $p$ and have same monomial as $q$ are those which can be obtained from $q$ by giving the same value to the pairs of successive numbers corresponding to some of its \emph{admissible ascents}. The set of such surjections is therefore homeomorphic to the set of subsets of admissible ascents. 

Consequently, except if $q$ has no admissible ascent, the alternated sum of the ordered partitions that have the same monomial as $q$ is $0$.
The only permutation that has no admissible ascent is $r(p)$, and $k(r(p))=n_v$ (since $r(p)$ is a permutation).
\end{proof}

\begin{lem}\label{lem:surj2}
For any surjection $p$, the following equality holds: \begin{equation} X(r(p))=X(\overline{p})^{-1}. \end{equation}
\end{lem}

\begin{proof}
For $i<j$: 

\begin{compactitem}
\item If $p(i)>p(j)$, then the power of $X_{ij}$ in $X(r(p))$ is $-1$ and the power of $X_{ij}$ in $X(\overline{p})$ is $1$.

\item If $p(i)\leq p(j)$, then the power of $X_{ij}$ in $X(r(p))$ is $1$ and the power of $X_{ij}$ in $X(\overline{p})$ is $-1$.
\end{compactitem}
\end{proof}

\section{Opening non-bicolorable maps}
\label{sec:nonBic}

We proved in \cref{thm:bijCut} that the opening algorithm, applied to bicolorable maps with dual-geodesic orientation, is actually a bijection with a certain family of unicellular blossoming maps (namely $\mathcal{O}$).
However, if the map is non-bicolorable, it is not possible to define its dual-geodesic orientation, because some adjacent faces may be at the same distances from the root face.

In this section, we extend the bijection to general maps (not necessarily bicolorable), by considering fractional orientation, in order to deal with the case of adjacent faces with the same label. This generalization is based on the work of Bouttier, di Francesco and Guitter in \cite{BoDiGu02}, that was later revisited by Albenque and Poulhalon in \cite{AlbPou15}. 

We define the \emph{face-doubled} (resp. \emph{vertex-doubled}) version of a map $m$, denoted $m^\parallel$ (resp. $m^\brokenvert$), as the map $m$ where each edge is replaced by $2$ adjacent copies of the edge (resp. divided into $2$ parts by adding a new vertex in the midst of it). 
The faces (resp. vertices) hereby created are called \emph{edge-faces} (resp. \emph{edge-vertices}).
Note that these are dual notions, so that $\dual{(m^{\parallel})}=(\dual{m})^\brokenvert$, and that $m^\parallel$ is bicolorable, while $m^\brokenvert$ is bipartite.
A \emph{face-doubled orientation} (resp. \emph{vertex-doubled orientation}) of a map $m$ is an orientation of $m^\parallel$ (resp. $m^\brokenvert$), with the additional constraint that no edge-face (resp. edge-vertex) is clockwise (resp. a sink).

The \emph{geodesic doubled-orientation} (resp. \emph{dual-geodesic doubled-orientation}) of a map $m$ is the vertex-doubled (resp. face-doubled) orientation of $m$ corresponding to the geodesic (resp. dual-geodesic) orientation of the bipartite map $m^{\brokenvert}$ (resp. the bicolorable map $m^{\parallel}$).

These orientations of the doubled map can alternatively be seen as \emph{fractional orientations} of the original map.
This means that for each edge of the map, instead of choosing either one orientation of the edge or the other, we choose a fractional combination of the $2$. In particular, we use \emph{half-orientations}, which means that each edge is either oriented in one direction, or in both, in which case it is called \emph{bi-oriented}. 
Note that half-orientations of a map are in bijection with doubled orientations of the doubled version of the map.
A half-orientation is called \emph{bipartite} (resp. \emph{bicolorable}) if its corresponding orientation in the vertex-doubled (resp. face-doubled) map is bipartite (resp. bicolorable).

Note that the geodesic half-orientation is bipartite. It can alternatively be defined by labeling all vertices by their distance to the root, and orienting each edge according to the labels of its adjacent vertices: toward the vertex with smaller label if the labels are different, or bi-oriented if the labels are equal.

The work of Propp can perfectly be applied to orientations of the doubled version of a map. Note however that the resulting lattice will contain orientations that are not doubled orientations, but that by definition, the minimum of the lattice will necessarily correspond to a doubled orientation, which makes it possible to transcribe it back to a fractional orientation of the original map.

\begin{thm}[Propp]\label{thm:fracPropp}
The transitive closure of the vertex-push (resp. face-flip) operation endows the set of bipartite (resp. bicolorable) orientations of the vertex-doubled (resp. face-doubled) version of a fixed map with a structure of distributive lattice, whose minimum is a vertex-doubled (resp. face-doubled) orientation, namely the geodesic (resp. dual-geodesic) doubled orientation.
\end{thm}

A face of a map with a fractional orientation is called \emph{clockwise} if it has no edge oriented completely counterclockwise. \cref{thm:fracPropp} leads to \cref{cor:fracPropp}:

\begin{cor}\label{cor:fracPropp}
The dual-geodesic half-orientation of a map is the unique bicolorable half-orientation of this map with no clockwise face.
\end{cor}

\begin{algorithm}[h!]
\caption{the fractional opening algorithm}
\label{alg:fracOp}
\begin{algorithmic} 
\REQUIRE A map $m$ embedded on a surface $\mathcal{S}$, rooted at a corner $c_0$, along with its dual-geodesic half-orientation.
\ENSURE A half-oriented blossoming embedded graph $b=\text{open}(m)$, embedded on $\mathcal{S}$.
\STATE Set $c=c_0$, $b=\emptyset$, and $E_V=\emptyset$ ($E_V$ is the set of visited edges). 
\REPEAT
\STATE $e=\nextE(c)$.
\IF{$e\notin E_V$ and $e$ is oriented toward $\vertex(c)$ or bi-oriented}
\STATE add $e$ to $E_V$
\STATE add $e$ to $b$
\STATE $c\leftarrow\nextF(c)$
\ELSIF{$e\notin E_V$ and $e$ is fully outgoing from $\vertex(c)$}
\STATE add $e$ to $E_V$
\STATE Add a bud to $b$ in place of $e$.
\STATE $c\leftarrow\nextV(c)$
\ELSIF{$e\in E_V$ and $e$ is fully oriented toward $\vertex(c)$}
\STATE Add a leaf to $b$ in place of $e$.
\STATE $c\leftarrow\nextV(c)$
\ELSIF{$e\in E_V$ and $e$ is outgoing from $\vertex(c)$ or bioriented}
\STATE $c\leftarrow\nextF(c)$
\ENDIF
\UNTIL {$c=c_0$}
\end{algorithmic}
\end{algorithm}

We now redefine the opening algorithm (see \cref{alg:fracOp}), by always considering that a bi-oriented edge is visited backward first. This roughly amounts to applying the classical opening algorithm to the face-doubled version of the map. Note that this algorithm can still be seen as the dual of a tour of a breadth-first-search exploration tree. The closing algorithm remains the same.

We redefine some properties to match fractional orientations. A unicellular map is called \emph{well-half-oriented} if, in a tour of the face, the first occurrence of any edge is either oriented backward or bi-oriented. The definition of a \emph{well-labeled} map is the same, with the additional rule that the labels of corners adjacent around a vertex and separated by a bi-oriented edge have to be equal.

The set of well-rooted well-labeled well-half-oriented blossoming unicellular maps, is denoted $\mathcal{OG}$. We count maps of $\mathcal{M}$ and $\mathcal{OG}$ by vertex degrees of any parity (unlike bicolorable maps, that we counted earlier only by even vertex degrees), so that for instance $M(\mathbf{z})=M(z_1,z_2,\cdots)=\sum_{m\in\mathcal{M}}\prod_{k=1}^{\infty}z_k^{v_{k}(m)}$.

We now state the generalization of \cref{thm:bijCut} to general maps.

\begin{thm}
The opening algorithm on a dual-geodesically half-oriented map is a weight-preserving bijection from $\mathcal{M}_g$ to $\mathcal{OG}_g$, whose reverse is the closing algorithm. Therefore, $\mathcal{M}_g(\mathbf{z})=\mathcal{OG}_g(\mathbf{z}).$
\end{thm}

\section{Acknowlegdement}
I would like to thank my supervisors Marie Albenque and Vincent Pilaud for many discussions about both the content and presentation of the present work. 

I also thank Guillaume Chapuy, Gilles Schaeffer, and Éric Fusy for fruitful discussion about their own work and the general state of the art.

\bibliographystyle{abbrv}
\bibliography{../../bibliography}

\end{document}